\documentclass[11pt,oneside,reqno]{amsart}
      \usepackage{amssymb,enumerate,dsfont,url}

      \theoremstyle{plain}
      \newtheorem{theorem}{Theorem}[section]
      \newtheorem{lemma}[theorem]{Lemma}
      \newtheorem{corollary}[theorem]{Corollary}
\newtheorem{proposition}[theorem]{Proposition}

      \theoremstyle{definition}

      \theoremstyle{remark}
      \newtheorem{remark}[theorem]{Remark}

      \makeatletter
      \def\@setcopyright{}
      
      \def\serieslogo@{}
      \makeatother

\usepackage{geometry}
\usepackage{marginnote}
\usepackage{verbatim}

\usepackage{color}

	\usepackage{mathtools}	

\newcommand{\R}{\mathbb{R}}
\newcommand{\C}{\mathbb{C}}

\renewcommand{\l}{\lambda}
\renewcommand{\d}{\delta}
\newcommand{\g}{\gamma}
\renewcommand{\a}{\alpha}
\newcommand{\e}{\varepsilon}
\newcommand{\s}{\sigma}
\renewcommand{\b}{\beta}

\newcommand{\lv}{\lvert}
\newcommand{\rv}{\rvert}
\newcommand{\lb}{\left(}
\newcommand{\rb}{\right)}

\newcommand{\GAi}{\textbf{(GA)$_1$}}
\newcommand{\GAii}{\textbf{(GA)$_2$}}
\newcommand{\map}[3] {#1:\, #2\, \rightarrow \, #3}

\newcommand{\Ai}{\operatorname{Ai}}

\renewcommand{\O}{\mathcal{O}}
\newcommand{\supp}{\operatorname{supp}}
\newcommand{\dopp}[1]{\mathds{#1}}
\renewcommand{\Re}{\operatorname{Re}}
\renewcommand{\Im}{\operatorname{Im}}
\newcommand{\E}{\mathbb{E}}
\newcommand{\U}{\mathcal{U}}
\newcommand{\lee}{\big\{}
\newcommand{\ree}{\big\}}
\newcommand{\lbb}{\big(}
\newcommand{\rbb}{\big)}
\newcommand{\Lip}{\textup{Lip}}
\newcommand{\F}{\widehat}
\newcommand{\lvb}{\big\lvert}
\newcommand{\rvb}{\big\rvert}
\renewcommand{\k}{\kappa}
\newcommand{\KAi}{K_{\textup{Ai}}}
\newcommand{\N}{\mathbb{N}}

\theoremstyle{plain}

\begin{document}

   \author{Kristina Schubert}
   \address{Inst. for Math. Stat., Univ. M\"unster, Orl\'{e}ans-Ring 10, 48149 M\"unster, Germany}
   \email{kristina.schubert@uni-muenster.de}

   \author{Martin Venker}
   \address{Fac. of Mathematics, Univ. Bielefeld, P.O.Box 100131, 33501 Bielefeld, Germany}
   \email{mvenker@math.uni-bielefeld.de}

   \title[Empirical Spacings of Unfolded Eigenvalues]{Empirical Spacings of Unfolded Eigenvalues}

   \begin{abstract}
We study random points on the real line generated by the eigenvalues in
unitary invariant random matrix ensembles or by more general repulsive
particle systems.
As the number of points tends to infinity, we prove convergence of the
empirical distribution of nearest neighbor spacings. 
We extend existing results for the spacing distribution in two ways. 
On the one hand, we believe the empirical distribution to be of more practical
relevance than the so far considered expected distribution.
On the other hand, 
we use the unfolding, a non-linear rescaling, which transforms the ensemble such that the  density of particles is asymptotically constant. 
This allows to consider all empirical spacings, where previous results were restricted to a tiny fraction of the particles. 
Moreover, we prove bounds on the rates of convergence. 
The main ingredient for the 
proof, a strong bulk universality result for correlation functions in the unfolded setting including optimal 
rates, should be of independent interest.   
\end{abstract}



 \thanks{The second author has been partially supported by 
the SFB 701.}


   \maketitle

\setcounter{equation}{0}

  \date{\today}
%


\section{Introduction}

The universal behaviour of eigenvalue statistics of random matrices has attracted much interest over the 
last 
decades. Although random matrices have already been studied by Wishart \cite{Wishart} in the 1920s, the development of the theory was promoted in the 1950s by 
Wigner \cite{Wigner},  who used eigenvalues of random matrices to model spectra of complex nuclei. Montgomery's  surprising discovery \cite{Montgomery} that zeros of 
the Riemann zeta function behave statistically as eigenvalues of random matrices, led to a further increase of interest. In 
recent years, the belief has emerged that  limit laws obtained in random matrix theory are also  ubiquitous in large systems of 
strongly correlated particles on the real line. One instance of this belief is the Bohigas-Giannoni-Schmit conjecture 
\cite{BGS}, stating that level spacings of quantum systems with classically chaotic motion should be given by random matrix laws. 
The 
ubiquity of certain limit laws has been established within Random Matrix Theory (RMT) as the universality phenomenon, which means 
that for large matrix 
sizes many eigenvalue statistics exhibit the same limit distributions for essentially different matrix models, provided these models share the same symmetry 
(e.g.~real-symmetry, complex-Hermiticity etc.). These universal limits usually arise for gap probabilities, spacing statistics or 
correlations of close eigenvalues in the bulk or at the edge of the spectrum if the mean spacing between consecutive eigenvalues is one. 

This paper is motivated by the following problem: Assume that  we consider a complicated (real-world) system. Based on a data set of real values obtained as a particular realization of that 
system, we want to study the 
appearance of universal RMT laws. The central question is how to detect such a universal behavior.

From a practical point of 
view, the easiest and most common statistic to consider is the empirical nearest-neighbor spacing distribution: for an ordered
$N$-tuple $x_1\leq\dots\leq x_N$ and an interval $I_N$ we denote by $\s(I_N,x)$  the counting measure 
of the 
nearest neighbor spacings in $I_N$ as
\begin{align}\label{defspacings}
 \s(I_N,x):=\sum_{x_j,x_{j+1}\in I_N} \d_{x_{j+1}-x_j}.
\end{align}
For a non-ordered $x$, we may define $\s(I_N,x):=\s(I_N,\bar{x})$, where $\bar{x}$ is an ordered tuple  build from the 
elements of $x$.
The empirical nearest-neighbor spacing distribution of $x$ in $I_N$ is then given by the probability measure
\begin{align*}
\hat{\s}(I_N,x):=\frac{1}{\int_0^\infty d\s(I_N,x)} \s(I_N,x).
\end{align*}
If $\int_0^\infty d\s(I_N,x)=0$, then we define $\hat{\s}(I_N,x)$ as an arbitrary probability measure on $\R$.

Here, $\hat{\s}(I_N,x)$ represents a histogram of the spacings from the data $x$ and has as such a high practical relevance. Indeed, 
histograms of spacings obtained in numerous (real-world) systems have been compared to limit distributions from RMT, ranging from 
level spacings in nuclear physics (see Mehta's classical book \cite{Mehta} and the references therein) to bus waiting times in 
certain Mexican cities \cite{KrbalekSeba}.

Despite its relevance, results on \textit{empirical} spacings in 
otherwise well-understood random 
matrix ensembles are surprisingly sparse. 
Most results in the literature are available for the \textit{expected} spacing distribution instead of empirical 
spacings. Briefly speaking, the expected spacing distribution is obtained by averaging over 
all realizations of $\s(I_N,x)$ or $\hat{\s}(I_N,x)$. A more formal definition and discussion is given below.
This preference of the expected spacing distribution is partly due to its
direct relation to correlation 
functions (marginals of fixed dimension).
For important classes of random matrices, these have particularly nice forms. This led to the convention of 
establishing local universality in terms of correlation functions.
However, to deduce strong universality results of empirical spacing statistics, quite strong forms of 
convergence of the correlation functions are needed, e.g.~uniformity in intervals \textit{growing} with $N$. These requirements are often not met 
by standard formulations of universality results.

In this paper, we prove the convergence of the empirical spacing distribution of unfolded eigenvalues or more general particles on 
the real line to a universal distribution, the Gaudin distribution. The unfolding basically consists of applying the limiting 
spectral distribution function to the eigenvalues/particles. This non-linear rescaling transforms the limiting spectral measure into 
the uniform measure on $[0,1]$ and allows for considering spacings of eigenvalues/particles from $I_N$ of macroscopic length, even 
the whole spectrum.
Our main theorem, Theorem \ref{thrmspacings}, states the uniform convergence of the distribution function of $\hat{\s}(I_N,x)$ towards the Gaudin distribution function $G$ in mean, i.e.~
\begin{align}
\label{first_statement_of_result}
\lim_{N \to \infty}\E_N\left( \sup_{s \in \mathbb R} \left| \int_0^s d\hat{\s}(I_N,x) - G(s) 
\right|\right)=  0.
\end{align}
We also obtain rates of convergence in terms of the length of the interval $I_N$.

Let us define the two models of this paper. 
In the first model, we consider the eigenvalues of Hermitian matrices from unitary invariant ensembles: Given functions $V,f: J\to 
\mathbb R$ on $J=[L_-,L_+]\cap\R$, which 
is a finite or infinite interval  ($- \infty  \leq L_- < L_+ \leq \infty$, for
precise assumptions see \GAi), we define a density on $\R^N$ by
\begin{align}
\label{density_P}
P_{N,V,f}(x):= \frac{1}{Z_{N,V,f}} \prod_{1 \leq i<j\leq N} \lv x_i-x_j\rv^2 e^{- N\sum_{j=1}^NV(x_j)+\sum_{j=1}^Nf(x_j)}
\dopp1_{J}(x_j).
\end{align}
If $f=0$, then we write $P_{N,V}$ instead of $P_{N,V,0}$. We will slightly abuse notation by not using different 
symbols for the 
measure and its density.
$P_{N,V,f}$ is the joint density of the eigenvalues of a random Hermitian matrix whose distribution has a density proportional to
$\exp(-\textup{Tr}(NV(M)+f(M)))$ with respect to the ``Lebesgue measure'' $dM=\prod_{j=1}^N \mathrm dM_{jj} \prod_{1 \leq j <k \leq 
N}
\mathrm dM_{jk}^R \mathrm dM_{jk}^I$ on the space of $N\times N$ Hermitian matrices $M$ with the property that all eigenvalues of 
$M$ lie in $J$. Most prominent in this class and arguably in 
the entire
RMT is the Gaussian unitary ensemble (GUE), which is obtained by choosing $V(t)=t^2, f=0$ and $J=\R$.

As a second model, we will  consider ensembles recently studied by G\"otze and the second author in \cite{GoetzeVenker}, which
generalize
 \eqref{density_P} by allowing for different interactions between particles.
Given $\map{Q,h}{\R}{\R}$ smooth (see \GAii\, for our assumptions), we define
\begin{align}\label{densityRPS}
 P_{N,Q}^h(x):=\frac{1}{Z_{N,Q}^h}\prod_{1 \leq i<j\leq N} \lv x_i-x_j\rv^2e^{-h(x_i-x_j)} e^{- 
N\sum_{j=1}^NQ(x_j)}.
\end{align}
We will call $P_{N,Q}^h$ a repulsive particle system. As $h$ is smooth, the densities $P_{N,Q}^h$ and $P_{N,V,f}$ vanish at the
same order if two particles approach each other. It is conjectured (cf. \cite{GoetzeVenker}) that the universal local limit 
laws
should only depend on the strength of the repulsion but be independent of the interaction at a non-zero distance and of the external
field $Q$.

Returning to spacings, let us review known results in RMT. The spacing distribution has been 
understood best for a particular ensemble of random unitary matrices, the circular unitary ensemble (CUE). It was introduced by 
Dyson \cite{Dyson} as a Haar distributed unitary matrix, which may be seen as having a uniform distribution on the group of $N\times 
N$ unitary matrices. The eigenvalues  lie  on the unit circle and have a joint density proportional to
\begin{align}
 \prod_{j<l}\lv e^{i\theta_j}-e^{i\theta_l}\rv^2,\label{CUE}
\end{align}
where $\theta_1,\dots,\theta_N\in[0,2\pi)$. Note the similarity in the interaction of \eqref{CUE} and \eqref{density_P}. It 
has been shown  by Katz and Sarnak \cite{KatzSarnak} that for some limiting function $G$ we have
\begin{align}
 \lim_{N\to\infty}\E_{\textup{CUE,N}}\sup_{s\in\R}\left\lvert \int_0^s\frac{2\pi}{\lv 
I_N\rv}d\s(I_N,(N/2\pi)\theta)-G(s)\right\rvert=\O\lb \lv I_N\rv^{-1/6+\e}\rb\label{KSresult}
\end{align}
for any $\e>0$, where $I_N=[0,2\pi N)$. 
Soshnikov \cite{Soshnikov} proved that  \eqref{KSresult} holds with a 
rate of $\O((\lv I_N\rv)^{-1/2+\e})$ for any $\e>0$, where $I_N\subset[0,2\pi N)$ is such that $\lv I_N\rv\to\infty$ as 
$N\to\infty$, in expectation as in \eqref{KSresult} and also almost surely. Furthermore, he 
proved that the fluctuations around the limit $G(s)$ are Gaussian.

In \eqref{KSresult},
\begin{align*}
 G(s):=\sum_{k\geq2}\frac{(-1)^k}{(k-1)!}\int_{[0,s]^k}\det\left[(S(z_i-z_j))_{1\leq i,j\leq k}\right]|_{z_1=0}dz_2\dots dz_k
\end{align*}
is the distribution function of the so-called Gaudin distribution and $S$ is the sine kernel,
\begin{align}
\label{sine_kernel}
 S(t-s):=\frac{\sin(\pi(t-s))}{\pi(t-s)}.
\end{align}
This probability measure had already been studied before in the physics literature (see e.g. \cite{Mehta} for references). In 
particular, the density $G'(s)$ is given as the second derivative of the Fredholm determinant of the integral operator with the 
sine kernel on $L^2((0,s))$.
Although this density does not seem to have a nice closed form expression, for many practical purposes it is sufficient to 
consider the so-called Wigner surmise instead, i.e.~$ p_2(s):=\frac{32}{\pi^2}s^2e^{-\frac\pi4 s^2}$ provides a good approximation 
to $G'(s)$.

Before we can further review   results on the spacing distribution for invariant ensembles, we first need to introduce the notion 
of the equilibrium measure in order to rescale the particles. It is 
well-known that
under very mild assumptions on $V$ and $J$, there is a measure $\mu_V$ on $\R$ with compact support that is the weak limit of the
expected empirical spectral distribution of $P_{N,V}$, i.e. for all continuous and bounded $\map{g}{\R}{\R}$
\begin{align}
 \lim_{N\to\infty}\E_{N,V}\frac1N\sum_{j=1}^Ng(x_j)=\int gd\mu_V.\label{LLN}
\end{align}
Here $\E_{N,V}$ denotes the expectation w.r.t.~$P_{N,V}$. The measure $\mu_V$ is the unique probability measure among all Borel
measures on  $J$ which minimizes the functional
\begin{align*}
 \mu\mapsto \int\int\lv t-s\rv^{-1}d\mu(t)d\mu(s)+\int V(t)d\mu(t).
\end{align*}
 Throughout our paper we will assume that $V$ is convex and real-analytic on a neighborhood
of the support of $\mu_V$ (see \GAi). Then $\mu_V$ has a density, which we will abusing notation also denote by
$\mu_V$, strictly positive in the interior of the support of $\mu_V$.

In invariant ensembles the results for the spacing distribution are by far weaker than for circular ensembles.
For instance, until the recent \cite{Schubert} by 
the first author, only the absolutely continuous intensity measure $\E\s(I_N,x)$ had been considered for these ensembles. We recall that if $x$ is random, then 
$\s(I_N,x)$ is a 
random measure and its intensity measure is defined as 
\begin{align}
 \int_B d\E\s(I_N,x):=\E \int_Bd\s(I_N,x)\label{intensity}
\end{align}
for any measurable set $B$ and $\E$ denotes expectation w.r.t.~the probability measure underlying the random variable $x$. We will 
call this measure the expected spacing measure.

To our knowledge, the first rigorous result on the spacing distribution for unitary invariant matrix ensembles 
is due to Deift et al.~\cite{Deiftetal99}. Let $a\in \supp(\mu_V)^\circ$ and $t_N>0$ such that $t_N\to\infty$ and $t_N/N\to0$ as 
$N\to\infty$. Setting $I_N=[a-t_N,a+t_N]$, it has been shown in \cite{Deiftetal99} following the method of Katz and Sarnak
that for real-analytic $V$ and $f \equiv 0$ in 
(\ref{density_P}), we have
\begin{align}
 \lim_{N\to\infty}\frac{1}{|I_N|\mu_V(a)}\int_0^s d\E_{N,V}\s\big(I_N,N\mu_V(a)x\big)
=G(s).\label{oldresult}
\end{align}

We observe that  \eqref{oldresult} clearly shows universality, as the r.h.s.~of \eqref{oldresult} does not depend on $V$.  Let us 
also remark that \eqref{oldresult} expresses the weak convergence of the distribution function of the asymptotically normalized 
expected spacing distribution to the Gaudin distribution. Furthermore, due to the continuity of $G$, \eqref{oldresult} actually 
holds uniformly in $s$.

An analogous result for certain Hermitian Wigner matrices was proved by Johansson in \cite{Johansson01}. Universality was proved 
for large classes of Wigner matrices and invariant ensembles in terms of the expected 
spacing measures by Bourgade, Erd\H{o}s, Schlein, Yau, Yin et al. (see \cite{BEY1} and the references therein). In 
particular, they show vague convergence of the asymptotically normalized 
expected spacing measure. Moreover, the  limiting distribution has been proved to be universal for large classes of real-symmetric, 
Hermitian or quaternionic 
self-dual random matrices as well as for general $\b$-ensembles, which are variants of \eqref{density_P} in which a parameter 
$\b>0$ replaces the exponent 2. The universal limit depends on the symmetry type or on $\b$, respectively.

Recently, there has been quite some interest in the distribution of a single spacing. This was initiated by  Tao, who 
proved in \cite{Tao13} that the Gaudin distribution is also the limit in law of a single spacing in the bulk of the spectrum, say 
of $x_{\lfloor N/2\rfloor+1}-x_{\lfloor N/2\rfloor}$. This shows that for expected spacings, an average over an increasing number 
of spacings as in \eqref{oldresult} is not necessary to obtain the limiting distribution. Tao proved this for Hermitian Wigner 
matrices and the 
result was later extended in \cite{ErdosYau} and \cite{Guionnet} to all symmetry classes and $\b$-ensembles. Moreover, the 
results in \cite{ErdosYau} and \cite{Guionnet} allow to consider statistics of the form (for ordered $x_i$)
\begin{align}
 \E_{N,V}\frac1{\lv I\rv}\sum_{i\in I} g(\mu_{V,i} N(x_{i+1}-x_i)),\label{singles}
\end{align}
where $\mu_{V,i}:=\mu_V(q_i)$ and $q_i$ is the $i/N$-quantile of the distribution $\mu_V$. 
The test function $g$ is assumed to be smooth and of compact support, thus determining vague convergence. The index set $I$ is 
assumed to contain only bulk eigenvalues in case of \cite{ErdosYau} and is arbitrary in \cite{Guionnet}. Both works show that in 
the large $N$ limit, \eqref{singles} become independent of $V$, thereby showing universality. However, to deduce convergence to the 
Gaudin distribution via \cite{Tao13}, the set $I$ has to fulfill $\e N\leq \min I\leq \max I\leq (1-\e)N$ for some $\e>0$.

It is instructive to compare the statements \eqref{singles} and \eqref{first_statement_of_result}. In \eqref{singles} as 
well as in \eqref{oldresult}, by first taking the expectation, the spacing statistics are averaged over all realizations, thus 
making it non-empirical. After that, this average is compared to the limiting quantity. In \eqref{first_statement_of_result}, the 
empirical statistic is first compared to the limit. The error of that comparison is then shown to vanish in $L^1$. It is further 
important to note that for statements like \eqref{first_statement_of_result}, the number of eigenvalues in the statistic 
necessarily has to increase with $N$. Indeed, the smooth function $G$ can only be approximated well by the step function 
$\int_0^\cdot d\hat{\s}(I_N,x)$ if the number of steps goes to infinity. The convergence of \eqref{singles} for a finite set $I$ is 
only possible since taking the expectation means averaging over infinitely many spacings (from all different realizations).

Let us turn to empirical spacings. In 
 \cite{Schubert} the first author of this work shows under certain assumptions on the Christoffel-Darboux kernel (see 
\eqref{def_kernel} for a definition of the kernel) that for unitary invariant ensembles
\begin{align}
 \lim_{N\to\infty}\E_{N,V}\sup_{s\in\R}\left\lvert \int_0^s\frac{1}{\mu_V(a)\lv 
I_N\rv}d\s(I_N,N\mu_V(a)x)-G(s)\right\rvert=0\label{result_Schubert}
\end{align}
with $I_N,a$ as in \eqref{oldresult}.
This result is quite analogous to \eqref{KSresult} for the CUE and was also 
shown in \cite{Schubert} for eigenvalues of real-symmetric 
and quaternionic self-dual ensembles.

However, this result has the drawback that the considered spacing 
statistics  are very inefficient in an empirical sense in that only 
a tiny fraction of the eigenvalues is used for the statistics. Indeed, the expected number of rescaled eigenvalues 
$N\mu_V(a)x_j$ in $I_N=[a-t_N,a+t_N]$ is about $2t_N$ and thus by the condition $t_N/N\to0$, the fraction of used eigenvalues even 
goes to 0, as $N$ gets large. This unsatisfying situation is due to the scaling of the eigenvalues
\begin{align}
 N\mu_V(a)x_j,\ j=1,\dots,N,\label{localization}
\end{align}
which we 
will call \textit{localized} around the point $a$. To obtain a universal limit, the eigenvalues have to be rescaled such that their
mean spacing is (asymptotically) one. This is achieved by the rescaling \eqref{localization} in a small vicinity of a point 
$a$ in the bulk of the spectrum, where the mean spacing is $(\mu_V(a)N)^{-1}+o(N^{-1})$. But as the equilibrium density $\mu_V(a)$ 
is generically not constant in $a$ (in contrary to the CUE), a linear 
rescaling such as \eqref{localization} does not allow to consider spacings of 
eigenvalues which are spread over an interval of macroscopic size. To overcome this problem, we will in this article unfold the 
eigenvalues with the distribution function of $\mu_V$. This natural non-linear rescaling produces an ensemble with constant 
equilibrium density.

\section{Statement of Results}\label{sec_statements}
Before we give a precise formulation of our results, we specify our assumptions for the ensembles \eqref{density_P} and \eqref{densityRPS}. 
Except for the support $J$, the class of ensembles $P_{N,Q}^h$ formally subsumes the unitary invariant ensembles $P_{N,V}$. 
However, our assumptions depend on the specific type of ensemble and therefore we will consider two different sets of assumptions.
For $P_{N,V,f}$ these assumptions are
\vspace{2em}

\begin{minipage}{0.08\linewidth}
 \textbf{(GA)$_1$}
\end{minipage}
\begin{minipage}{0.85\linewidth}
 \begin{enumerate}
 \item $V,f:J\rightarrow \R$ are real analytic, $J=[L_-, L_+] \cap \R$ with $-\infty \le L_-<L_+ \le \infty$. \vspace{4pt}
\item $\inf_{t\in J}V''(t)>0$. (strict convexity)\vspace{4pt}
\item $\lim_{\lv x\rv\to\infty}V(x)-cf(x)=\infty$ for some $c>0$ in case that $J$ is unbounded.\vspace{4pt}
\item $L_-,L_+$ do not belong to the support of the equilibrium measure $\mu_V$. 
\end{enumerate}
\vspace{4pt}
If $f=0$, then (2) can be replaced by the weaker condition
\begin{enumerate}
 \item[(2')] $V'$ is strictly increasing.
\end{enumerate}
\end{minipage}\\

\vspace{2em}

For the ensembles $P_{N,Q}^h$, we make the following assumptions.
\vspace{2em}

\begin{minipage}{0.08\linewidth}
 \textbf{(GA)$_2$}
\end{minipage}
\begin{minipage}{0.85\linewidth}
 \begin{enumerate}
 \item $\map{Q,h}{\R}{\R}$ are real-analytic and symmetric around 0.\vspace{4pt}
\item $\a_Q:=\inf_{t\in\R}Q''(t)>0$.
\item $h$ is  a Schwartz function with exponentially fast decaying Fourier transform.
\end{enumerate}
\end{minipage}\\
\vspace{2em}

Under \GAii\,, it has been shown in \cite{GoetzeVenker} that for given $h$, for each $Q$ with $\a_Q$ large enough (depending on $h$), 
there
exists a probability measure $\mu_{Q}^h$ such that \eqref{LLN} holds with $\mu_Q^h$ replacing $\mu_V$ and $\E_{N,Q}^h$ replacing
$\E_{N,V}$. Furthermore, $\mu_Q^h$ has also the generic form \eqref{def_rho} for $V:=Q+h*\mu_Q^h$, where $*$ denotes 
convolution.\\

\begin{remark}[On the assumptions]
One could without substantial changes also introduce the external field $f$ to the model $P_{N,Q}^h$. For the sake of notational convenience this is not done 
here. The model could also be studied on an interval $J$, which might not be the whole line. This would 
require a condition analogous to \GAi~(4) for $P_{N,Q}^h$.  For repulsive particle systems  the equations \eqref{MRS 
numbers}, which determine the endpoints of the equilibrium measure, depend on the equilibrium measure itself, which makes checking 
an analog of \GAi~(4) more complicated than for invariant ensembles. Nevertheless, we will show in our analysis that a truncation 
to large enough $J$ only produces an asymptotically negligible error.
\end{remark}

To define the scaling of the particles, 
for $P_{N,V,f}$  let $F_V$ denote the distribution function of the equilibrium measure $\mu_V$ and consider the \textit{unfolded eigenvalue}
\begin{align*}
 \tilde{x}_i:=NF_V(x_i).
\end{align*}
If $x$ is a random configuration sampled from $P_{N,V,f}$, then $\tilde{x}$ is a point process on
$[0,N]$ with asymptotically constant density and mean spacing 1, at least in the bulk.
If $x$ is distributed according
to the repulsive particle system $P_{N,Q}^h$, let $F_Q^h$ denote the distribution function of $\mu_Q^h$ and define

 \begin{align*}
 \tilde{x}_i:=NF_Q^h(x_i).
\end{align*}
Now, our main theorem on the spacing distribution reads as follows.

\begin{theorem}\label{thrmspacings}
Let $I_N\subset[0,N]$ be
a sequence of intervals with 
$$\liminf_{N\to\infty}\dfrac1N\textup{dist}(I_N,\{0,N\})>0\ \text{ and }\ \lim_{N\to\infty}\lv I_N\rv\to\infty.$$ 
\begin{enumerate}
 \item  Let $P_N$ be either $P_{N,V,f}$ satisfying \GAi\, or $P_{N,Q}^h$  satisfying \GAii\, with $\a_Q$ large 
enough
(this depends on $h$ only) and $h$ negative-definite. Then for any $\e>0$
\begin{align}
 \E_N\left( \sup_{s \in \mathbb R} \left| \int_0^s d\hat{\s}(I_N,\tilde{x}) - G(s) 
\right|\right)=  \O
\left( \lv I_N\rv^{-\frac{1}{4} +\varepsilon} \right).\label{thrmeq2}
\end{align}
\item If $P_N=P_{N,Q}^h$ with $h$ not necessarily negative-definite, then for $\a_Q$ large enough \eqref{thrmeq2} holds with the 
$\O$-term being replaced by $o(1)$.
\end{enumerate}
In (1) and (2), $\hat{\s}(I_N,\tilde{x})$ can be replaced by $\lv I_N\rv^{-1}\s(I_N,\tilde{x})$ without altering the result.

\end{theorem}
\begin{remark}\noindent
\begin{enumerate}
\item In the language of mathematical statistics, \eqref{thrmeq2} implies that
$\hat{\s}(I_N,\tilde{x})$ is an asymptotically consistent estimator for the Gaudin distribution, considered in the Kolmogorov 
metric. Note that similar results on the expected 
spacing distribution like \eqref{oldresult} only show the asymptotic unbiasedness.
For applications, it 
would be 
favorable to have the unfolding with $F:=F_V$ or $F:=F_Q^h$ being replaced by an unfolding based solely on the empirical values. We 
remark that using the empirical distribution function as a naive estimator for $F$ would result in a non-random $\tilde{x}$. 
However, we expect that a smoothed empirical distribution function should yield the desired. In fact, this is a typical 
procedure in statistical analysis. It will be considered in a future work.
\item To our knowledge, Theorem \ref{thrmspacings} is the first result on rates of convergence to the Gaudin distribution for 
invariant ensembles.  For the simpler CUE a rate of $\O(\lv I_N\rv^{-1/6+\e})$ for any
$\e>0$ was derived in \cite{KatzSarnak} resp.~a rate of $\O(\lv I_N\rv^{-1/2+\e})$ was shown in \cite{Soshnikov}. 
Numerical experiments  (cf. 
\cite{KriecherbauerSchubert}) suggest that the optimal rate for the GUE is $\lv I_N\rv^{-1/2}$,
possibly with some logarithmic factor. The 
dependence of the rate on $\lv I_N\rv$ reflects the fact that necessarily a growing number of empirical spacings has to be 
considered in order to obtain convergence.

\item
For negative-definite $h$, an exact representation of $P_{N,Q}^h$ in terms of determinantal ensembles will be derived in Section 
\ref{sec_rep_part}, which allows to transfer rates of convergence from the unitary invariant ensembles to the repulsive particle 
systems. For general $h$, only convergence can be shown, see Remark \ref{remark_Vitali} for more details. 
On the other hand, if $h$ is positive-definite, then it suffices to have $\a_Q>\sup_{t\in\R}-h''(t)$.  
\end{enumerate}
\end{remark}

Abandoning any rate of convergence, we an also deduce the result of Theorem \ref{thrmspacings} for $I_N=[0,N]$.
\begin{corollary}\label{cor_spacings}
With $P_N$  as in Theorem \ref{thrmspacings} (1) or (2), we have
\begin{align}
 \lim_{N\to\infty}\mathbb E_N\left( \sup_{s \in \mathbb R} \left| \int_0^s d\hat{\s}([0,N],\tilde{x}) - G(s)
\right|\right)= 0.\label{thrmeq1}
\end{align}
\end{corollary}

\begin{remark}
 Note that in \eqref{thrmeq1}, edge spacings are included. Although correlations between
eigenvalues at the edge are not given by the sine kernel, the number of edge spacings is relatively small and thus does not change
the limit distribution.
\end{remark}

The next corollary shows a much better rate of convergence for the expected spacing distribution. We believe that this rate is almost optimal.
\begin{corollary}\label{cor_intensity}
Let $P_N, I_N$ be as in Theorem \ref{thrmspacings} (1). Then for any $\e>0$
 \begin{align*}
\sup_{s\in\R}\left\lv\int_0^s\frac{1}{\lv I_N\rv}d\E_N\s(I_N,\tilde{x}) - G(s)\right\rv=\O\lb \lv
I_N\rv^{-1+\e}\rb.
       \end{align*}
\end{corollary}

\begin{remark}
 Recall from the discussion around \eqref{singles} that a similar result should also hold for intervals with $\lv I_N\rv$ of order 
1, probably with rate $\O(N^{-1+\e})$. As Corollary \ref{cor_intensity} is merely a byproduct of an analysis which necessarily 
deals with growing intervals, we will not pursue this here.
\end{remark}

A major ingredient to the proof of Theorem \ref{thrmspacings} is a strong form of  bulk universality for correlation functions,  
which should be of
independent interest. To state it, let us recall the notion of correlation functions. 

For a probability measure $P_N(x)dx$ on $\R^N$, invariant under permutations of its arguments, the $k$-th correlation function is 
the map $\map{R_{N}^k}{\R^k}{\R}$, 
\begin{align}
 R_N^k(t_1,\dots,t_k):=\frac{N!}{(N-k)!}\int_{\R^{N-k}} P_N(t_1,\ldots, t_N)dt_{k+1}\dots dt_N.\label{correlationfunctions}
\end{align}
Note that in previous works of the second author, the correlation functions are defined as the $k$-dimensional marginal densities 
of $P_N$ and therefore differ from \eqref{correlationfunctions} by the factor $N!/(N-k)!$. The definition in 
\eqref{correlationfunctions} is more convenient for dealing with sums and will be used throughout this article. 

A crucial fact for many computations and universality proofs is that unitary invariant ensembles are determinantal, that is
\begin{align}
 R_{N,V,f}^k(t)=\det(K_{N,V,f}(t_i,t_j))_{1\leq i,j\leq k},\label{determinantal_relations}
\end{align}
such that the analysis boils down to studying the so-called Christoffel-Darboux kernel
\begin{align}
 K_{N,V,f}(t,s):=\sum_{j=0}^{N-1}p_{j,N}(t)p_{j,N}(s)e^{-N/2(V(t)+V(s))+1/2(f(t)+f(s)},\label{def_kernel}
\end{align}
where $p_{j,N}, j=0,1,\dots$ are the orthonormal polynomials with positive leading coefficients to the weight 
$e^{-(1/2)(NV(t)-f(t))}$ on $J$.

 To formulate a certain uniformity in the field $f$, let for a domain $D\subset\C$, $(X_D, \|\cdot\|_D)$ 
denote the 
Banach space of functions $\map{f}{D}{\C}$ which are analytic, bounded and real-valued on $D\cap \R$. Here $\|\cdot\|_D$ denotes 
the sup-norm on $D$.

\begin{theorem}\label{universality}
Let $I_N\subset[0,N]$ be such that $\liminf_{N\to\infty}\textup{dist}(I_N,\{0,N\})/N>0$. 
\begin{enumerate}
 \item Let $V,f$ satisfy \GAi.
 Then
\begin{align*}
\frac{1}{N \mu_V(F_V^{-1} (\frac{t}{N}))} K_{N,V,f} \left( F_V^{-1}\left(\frac{t}{N}\right),F_V^{-1}\left(\frac{s}{N} \right) 
\right)=
\frac{\sin (\pi (t-s))}{\pi (t-s)} + \mathcal O\left(  \frac{1}{N}\right),
\end{align*}
where the error term is uniform for $t,s \in I_N$.
\item Assume in addition to (1) that  $J$ is compact and let $0<\eta<1$. Then there is a complex domain 
$D\supset J$ such that (1) holds uniformly for $f\in X_D$ with $\|f\|_D\leq
N^\eta$ if $\O(1/N)$ in (1) is replaced by $\O(N^{\eta-1})$.
\item Let $R^k_N$ be the $k$-th correlation function of $P_N$ with $P_N$ as in Theorem \ref{thrmspacings} (1). Then, 
abbreviating $\hat{t}_j:=F_V^{-1}(t_j/N)$ in the
unitary invariant case and $\hat{t}_j:=(F_Q^h)^{-1}(t_j/N)$ in the case $P_N=P_{N,Q}^h$, as well as writing $\mu$ for $\mu_V$ and 
$\mu_Q^h$, respectively, we have for any $\e>0$
\begin{align*}
\frac{1}{N^k \prod_{j=1}^k\mu(\hat{t}_j)} R_N^k(\hat{t}_1,\dots,\hat{t}_k)=\det\left[\frac{\sin (\pi (t_i-t_j))}{\pi
(t_i-t_j)}\right]_{1\leq i,j\leq k}+\O\lb N^{-1+\e}\rb
\end{align*}
with the $\O$ term being uniform for $t_1,\dots,t_k\in I_N$. If $P_N=P_{N,V,f}$, then the statement is valid for $\e=0$. For $P_N$ as 
in Theorem \ref{thrmspacings} (2), the statement is valid with the $\O$ term replaced by $o(1)$.
\end{enumerate}
\end{theorem}

\begin{remark}
The convergence of the Christoffel-Darboux kernel of some determinantal ensemble to the sine kernel is a very typical result in RMT and the content of numerous papers in the field. We will mention here only very few seminal papers and refer to \cite{handbookKuijlaars} for an overview instead. A first universality proof was given by Pastur and Shcherbina \cite{PasturS97} (see also \cite{PasturS08}) for sufficiently smooth $V$. Deift et al.~\cite{Deiftetal99} showed universality for real-analytic $V$ using Riemann-Hilbert techniques and more recently  Levin and Lubinsky \cite{LL08} established it under very mild assumptions on $V$ using complex analysis.

In the existing literature, the kernel is usually considered in the localized scaling, that is $K_N(a+\frac{t}{N\mu(a)},a+\frac{s}{N\mu(a)})$, where $a$ is in the bulk of the spectrum and $\mu$ is the limiting spectral density. Then $t$ and $s$ are typically assumed to lie in some fixed compact set, that is their distance is bounded in $N$. In the recent \cite{KSSV}, convergence is shown under the assumption $\lv t-s\rv=o(N)$ with the rate $\O((1+\lv t\rv+\lv s\rv)/N)$, which is optimal in the localized scaling.
To our knowledge, Theorem \ref{universality} is the first version of bulk universality for unitary invariant ensembles which does
not require all eigenvalues to lie in a vicinity of some point $a$. Moreover, the rate in part (1) of the theorem is optimal, as can be seen for instance in \eqref{bulk-formula}.
\end{remark}

Theorem \ref{thrmspacings} will be deduced from the following more general result.
\begin{theorem}\label{thrmgeneral}
Let for each $N$, $I_N\subset [0,\infty)$ be an interval and $P_N(x)dx$ be a probability measure on $\R^N$, invariant under 
permutation of 
the
coordinates. Let  $R_N^k$ denote the $k$-th correlation function of $P_N(x)dx$, defined in \eqref{correlationfunctions}. 
Further 
let $C_0>0$ denote a constant such that the following conditions are satisfied: 
\begin{enumerate}
 \item For all  $N\in \mathbb N$ we have
\begin{align*}
        \sup_{t_1,\dots,t_k\in I_N}\left\lvert\frac{1}{N^k} R_N^k(t_1/N,\dots,t_k/N)\right\rvert\leq C_0^k.
       \end{align*}
\item There exists $\k_N>0$ with $\lim_{N\to\infty}\k_N=\infty$ such that for all $N\in \mathbb N$ we have
\begin{align*}
   \sup_{t_1,\ldots,t_k  \in I_N}
		\left| \frac{1}{N^k}R_N^k(t_1/N, \ldots, t_k/N) - S_k(t_1,\ldots,t_k) \right| \leq   k^{k/2+1}\cdot C_0^k 
\cdot\frac{1}{\k_N}, 
\end{align*}
where for  $S$ as in \eqref{sine_kernel}, $S_k$
is given by
\begin{align*}
       S_k(t_1,\ldots,t_k) \coloneqq  \det \left[ S(t_i-t_j)\right]_{1 \leq i,j \leq k}.
      \end{align*}
\end{enumerate}
Then for any $\e>0$
\begin{align*}
 \E_N\left( \sup_{s \in \mathbb R} \left| \int_0^s  d\hat{\s}(I_N,Nx) - G(s) 
\right|\right)=  \O
\left( \min(\lv I_N\rv,\k_N)^{-\frac{1}{4} +\varepsilon} \right).
\end{align*}
The same result holds with $\hat{\s}(I_N,Nx)$ being replaced by $\lv I_N\rv^{-1}\s(I_N,Nx)$.
\end{theorem}

The preceding theorem can also be used to show universality of the spacing distribution in the localized scaling.
\begin{corollary}\label{Localized Scaling}
 Let $P_N$ be as in Theorem \ref{thrmspacings} (1). Set $\mu=\mu_V$ if $P_N=P_{N,V,f}$ and $\mu=\mu_Q^h$ if $P_N=P_{N,Q}^h$. Let 
$a$ be 
such that $\mu(a)>0$ and $I_N\subset J$ be a symmetric interval with center $a$ and $\lim_{N\to\infty}\lv I_N\rv=\infty$ but 
$\lim_{N\to\infty} \lv I_N\rv/N=0$. Then
\begin{align*}
  \mathbb E_N\left( \sup_{s \in \mathbb R} \left| \int_0^s d\hat{\s}(I_N,N\mu(a)x) - G(s) 
\right|\right)
=\O
\left( \min(\lv I_N\rv,N/\lv I_N\rv)^{-\frac{1}{4} +\varepsilon} \right).
\end{align*}
If $P_N$ is as in Theorem \ref{thrmspacings} (2), the last statement is valid with the $\O$ term being replaced by $o(1)$. In either
case, the statements remain valid if  $\hat{\s}(I_N,N\mu(a)x)$ is replaced by $(\lv 
I_N\rv\mu(a))^{-1}\s(I_N,N\mu(a)x)$.
\end{corollary}
\begin{remark}
 Note that  for $\lv I_N\rv\gg \sqrt{N}$,  Theorem \ref{thrmspacings} provides a better rate of convergence for the 
unfolded 
particles than Corollary \ref{Localized Scaling} for the localized particles. 
\end{remark}
We will finish this section with some concluding remarks. The repulsive particle systems $P_{N,Q}^h$ appeared first in a more 
general setting with many-body interactions in \cite{BPS}, where under some convexity condition on the additional interaction 
results on the equilibrium measure were announced. Further results associated with global asymptotics in such models can be found in 
\cite{Borotetal} and \cite{Chafaietal}. Local  bulk universality was proved for $P_{N,Q}^h$ in \cite{GoetzeVenker} and for the $\b$ 
variants in \cite{Venker13}. Edge universality and fine asymptotics of the largest particle have been considered for $P_{N,Q}^h$ in 
the recent \cite{KV}.
\vspace{2em}

The  paper is organized as follows: After a brief outline, the proof of Theorem \ref{thrmgeneral} will
be given in Section \ref{sec:proof_th_general}. Here we follow the method developed by Katz
and Sarnak in \cite{KatzSarnak}, streamlining and optimizing  in order to obtain better rates of convergence. Theorem
\ref{universality} (1) and (2) will be proved in Section \ref{sec:proof_universality}. The proof of Theorem
\ref{universality} (3) for the repulsive particle systems relies on a non-trivial reduction to the unitary invariant case. An 
outline of the method from \cite{GoetzeVenker} and the proof of Theorem
\ref{universality} (3) are contained in Section \ref{sec_rep_part}. The proofs of Theorem \ref{thrmspacings} and the corollaries are given in 
Section \ref{sec_remaining_res}.


\section{Investigation of the Spacing Distribution -- Proof of Theorem \ref{thrmgeneral}} \label{sec:proof_th_general}
Theorem \ref{thrmgeneral} will be proved first with the asymptotic and non-random 
normalization $\lv I_N\rv$, the case of the exact but random normalization will be discussed at the end of the proof. Moreover, let 
us note that in this case the statement of Theorem \ref{thrmgeneral} is trivial if $\lv I_N\rv$ is bounded in $N$. Hence, we will 
from now on assume that $\lv I_N\rv\to\infty$, as $N\to\infty$. 

Let us furthermore make some notational remarks. By $C$ we denote 
absolute positive constants that may change from line to line. Finally, note that we often suppress certain $N$-dependencies.

A major disadvantage of $\s(I_N,x)$ is its dependence on the ordering $x_1\leq\dots\leq x_N$, which prevents an 
efficient use of correlation functions. This problem can be circumvented by using the measures
\begin{align*}
 	\gamma^k(I_N,x)\coloneqq
	 \frac{1}{|I_N|}
	\sum_{\substack{i_1 < \ldots < i_k,\\ x_{i_1},x_{i_k}\in 
	I_N}}\delta_{(\max_{1 \leq j \leq k} x_{i_j} - \min_{1 \leq j \leq k} x_{i_j})},
\end{align*}
which are symmetric and fulfil the relations 

\begin{align}
\frac{1}{\lv I_N\rv}\int_{0}^s  \, d\sigma (I_N,x)&= \sum_{k=2}^N (-1)^k \int_0^s
\, d \gamma^k(I_N,x), \quad N \in \mathbb N,\label{ks_sums2} \\
(-1)^m \frac{1}{\lv I_N\rv} \int_0^s d\sigma(I_N,x) &\leq (-1)^m\sum_{k=2}^m (-1)^k \int_0^s d \gamma^k(I_N,x), 
\quad m \leq N.
\label{ks_alt_sums}
\end{align}

The proof of Theorem \ref{thrmgeneral} consists of three steps. The first step establishes the point-wise convergence
 \begin{align}
\label{con_point_wise} 
 \lim_{N \to \infty}\mathbb E_N\left(   \int_0^s \frac{1}{\lv I_N\rv} d\sigma(I_N,Nx)\right)= G(s)
 \end{align}
and bounds the difference of $\int_0^s \frac{1}{\lv I_N\rv} d\sigma(I_N,Nx)$ and $G(s)$ in terms of the variances of the 
$\g^k$'s. In the second step, these variances are estimated. From 
this we can deduce a bound on 
\begin{align}
\label{main_proof_step2}
\mathbb   E_N\left( \left| \frac{1}{|I_N|}\int_0^s d\sigma (I_N, Nx) - G(s) \right| \right)
\end{align}
(see Corollary \ref{koro_erw_2}).
 
The difference between (\ref{main_proof_step2}) and the quantity to be estimated in Theorem \ref{thrmgeneral}, is that we need to 
take the supremum over all $s$ before taking the expectation.   
This issue is addressed by considering (\ref{main_proof_step2}) at a (growing) number of nodes $s_i$ rather than at a single $s$. 
The respective results are provided in Section \ref{sec_sup}.\\
\medskip

Before we turn to the proof of Theorem \ref{thrmgeneral}, we note an important estimate for power sums. We will frequently encounter 
 sums of the form $\sum_{k=2}^La_kz^k$ with $L$ and $z$ growing in $N$ and with 
different sequences $(a_k)$. To provide a unified and efficient treatment of these sums, the following simple lemma will prove 
useful.

For an entire function $f$, recall that $f$ is said to be of finite order if the inequality
\begin{align*}
 \max_{\lv z\rv\leq r}\lv f(z)\rv<e^{r^\kappa}
\end{align*}
holds for all $r$ large enough and some $\kappa<\infty$. The greatest lower bound of all such $\kappa$ is called order of $f$. If 
$f$ has a power series expansion $\sum_{k=0}^\infty a_kz^k$, then the order $\rho$ can be found via
\begin{align}\label{order}
 \rho=\limsup_{k\to\infty}-\frac{k\log k}{\log \lv a_k\rv}.
\end{align}
If $f$ has finite order $\rho$, then it is said to be of finite type, if 
\begin{align*}
 \max_{\lv z\rv\leq r}\lv f(z)\rv<e^{\zeta r^\rho}
\end{align*}
holds for $r$ large enough and some finite $\zeta$. The greatest lower bound $\nu$ of all such $\zeta$ is called the type of $f$ 
and can be determined via
\begin{align}\label{type}
(\nu e\rho)^{1/\rho}=\limsup_{k\to\infty}k^{1/\rho}\lv a_k\rv^{1/k}.
\end{align}

\begin{lemma}\label{lemma_sums}
 Let $f(z)=\sum_{k=0}^\infty \lv a_k\rv z^k$ be an entire function of order at most 2 and finite type. Let $(L_N)_N, 
(\d_N)_N,(M_N)_N$ be sequences with $L_N>0$, $L_N=o(\log M_N)$ and $0<\d_N<\sqrt{\log M_N}$. Then
\begin{align*}
 \left\lv \sum_{k=0}^{L_N}a_k\delta_N^k\right\rv=\O(M_N^\e)
\end{align*}
for any $\e>0$.
\end{lemma}

\begin{proof}
 Let $\e>0$ and $\nu$ denote the type of $f$. For any $K>1$, we have the trivial estimate
\begin{align*}
 \sum_{k=0}^{L_N}\lv a_k\rv\d_N^k\leq K^{L_N}\sum_{k=0}^{L_N}\lv a_k\rv\left(\frac{\d_N}{K}\right)^k\leq K^{L_N}f(\d_N/K)\leq 
CK^{L_N}e^{\nu\frac{\d_N^2}{K^2}}\leq CK^{L_N} M_N^{\frac \nu{K^2}}.
\end{align*}
Choosing $K=K(\e)$ large enough, the lemma is proved. Here we used that $K^{L_N}$ is of subpolynomial growth in $M_N$, due 
to our assumption on $L_N$.
\end{proof}

\subsection{The convergence of $\mathbb E_N\left(   \int_0^s \frac{1}{\lv I_N\rv} d\sigma(I_N,Nx)\right)$}
We turn to the proof of  (\ref{con_point_wise}).
For $s>0$ and  $t_1,\ldots,t_k$, we denote by $\chi_{s,I_N}$ the function	
	\begin{align*}
	\chi_{s,I_N}(t_1,\ldots,t_k)\coloneqq \dopp{1}_{(0,s)} \left(\max_{i=1,\ldots,k} t_i 	-\min_{i=1,\ldots,k} t_i \right) 
\prod_{i=1}^k \dopp{1}_{I_N} (t_i).
	\end{align*}
To select certain entries of a vector $x=(x_1,\ldots,x_N) \in\mathbb R^N$ we use the notation
\begin{align*}
 	x_T \coloneqq (x_{i_1},  \ldots,  x_{i_k}), \quad  T=\{i_1, \ldots, i_k\}, \,  1 \leq i_1 < \ldots < i_k \leq N.
	\end{align*} 
With this notation we can rewrite
	\begin{align}\label{Darst:gamma_N}
	\int_0^s d\gamma^k(I_N,Nx)= \frac{1}{|I_N|}\sum_{T \subset \{1,\ldots, N\}, |T|=k} \chi_{s,I_N} (Nx_T)
	\end{align}
and we obtain 
	\begin{align}
&\mathbb E_{N} \left(\int_0^s d\gamma^k(I_N,Nx) \right)\nonumber\\
&= \frac{1}{|I_N|k!N^k}\int  \chi_{s,I_N} (t_1,\ldots,t_k)R_N^k(t_1/N,\dots,t_k/N) dt_1 \ldots 
dt_k.\label{Expect_gamma}
\end{align} 

%
The following lemma establishes the convergence of the terms $\mathbb E_{N} \left(\int_0^s d\gamma^k(I_N,Nx) \right)$  and further  
provides a useful estimate for the proof of 
Theorem \ref{thrmgeneral}. The proof is essentially contained in \cite{Schubert} and we revisit the arguments in order to 
adjust them to the 
current setting.
\begin{lemma}\label{pointwise} 
Let the conditions of Theorem \ref{thrmgeneral} be satisfied. 
\begin{enumerate}
\item For $k \geq 2$   we have 
 \begin{align} 
 \mathbb E_{N}&\left(\int_0^s   d\gamma^k(I_N,Nx) \right) = \int_{0 \leq z_2 \leq \ldots \leq z_k \leq s}  S_k(0, z_2,\ldots,z_k) 
		dz_2 \ldots dz_k \nonumber \\
		&\quad + \frac{1}{(k-1)!} s^k \mathcal{O}\left(\frac{1}{|I_N|}\right) C_0^k + s^{k-1}   C_0^k 
\frac{k^{k/2+1}}{(k-1)!} \frac{1}{\k_N},\label{absch_lim_gamma}
\end{align}
where the $\O$ term is uniform for $s\in\R$ and $k\in\mathbb{N}$.
	\item
	For $A_N:=\min(\lv I_N\rv,\k_N)$,  $0<s:=s_N, s_N=\O(\sqrt{\log A_N})$ and 
$L:=L_N\in\mathbb{N}$ such 
that $\sqrt{\log A_N}=o(L_N)$ and $L_N=o(\log A_N)$, we 
have for any $\e>0$
\begin{align*} 
 			 \left| G(s) -\frac{1}{\lv I_N\rv} \int_{0}^s  \,
 			d\sigma (I_N,Nx) \right| 
 			 \leq &\sum_{k=2}^L \left|
 			\mathbb{E}_{N}\left( \int_{0}^s  \, d\gamma^k(I_N,Nx) \right) -
 			\int_{0}^s  \, d\gamma^k(I_N,Nx) \right|
 			\\ 
 			&\quad + \mathcal{O}\left(\frac{1}{A_N^{1-\e}}\right).
\end{align*}			
\end{enumerate}
\end{lemma}
\begin{proof}  
In order to prove (1), we consider (\ref{Expect_gamma}) and use the uniform convergence of the correlation functions in (2) of 
Theorem \ref{thrmgeneral},
i.e. we use that $\frac{1}{N^k}R_N^k(t_1/N, \ldots, t_k/N)=  S_k(t_1,\ldots,t_k) +  k^{k/2+1}\cdot C_0^k 
\cdot
\frac{1}{\k_N}$ uniformly on $I_N$. 
We further use the obvious estimate
\begin{align}
\label{int_chi}
\frac{1}{|I_N|k!}\int_{\R^k}\chi_{s,I_N} (t) dt \leq 
 \frac{s^{k-1}}{(k-1)!}
\end{align}
to obtain
\begin{align}
			\mathbb E_{N} \left(\int_0^s d\gamma^k(I_N,Nx) \right)
		 =  \frac{1}{|I_N|k!}\int_{\R^k}  \chi_{s,I_N} 			(t)S_k(t) dt
		  +   
		\frac{s^{k-1}}{(k-1)!}k^{k/2+1}   C_0^k\cdot\frac{1}{\k_N}.\nonumber 
\end{align} 
The translational invariance of $S_k$ and the change of variables $z_1=t_1, z_i=t_i-t_1, i=2,\ldots, k$ lead to
	\begin{align*}
		  \frac{1}{|I_N|k!}\int_{\R^k}  \chi_{s,I_N} 			(t)S_k(t) dt= 
		\frac{1}{|I_N|} 
		\int_{0 \leq z_2 \leq \ldots \leq z_k \leq s \atop  z_1\in I_N, \, z_1+z_j \in I_N, j=2,\ldots}    S_k(0, 
z_2,\ldots,z_k) 
		 dz_1 \ldots 			dz_k 
		.
	\end{align*}
	Observe that in the latter integral, we integrate over $z_1$ from $I_N$ except for an interval which has at 
most length $s$. The estimate in (1) is then obvious from $\sup_{t \in I_N^k}|S_k(t)|\leq C_0^k$ (see Theorem \ref{thrmgeneral} 
(1)). Observe that statement (1) together with \eqref{ks_alt_sums} already implies \eqref{con_point_wise}.

  To show (2), we first introduce the notation 
\begin{align}
E(s,k)\coloneqq  \int_{0 \leq z_2 \leq \ldots \leq z_k \leq s}  S_k(0, z_2,\ldots,z_k) 
		dz_2 \ldots dz_k. \label{def_E}
\end{align}
The idea is to use \eqref{ks_alt_sums} and bound $G(s)$ and $\int_{0}^{s}  \,
 			d\sigma (I_N,Nx)$ from above and from below by alternating sums over $E(s,k)$ and $\int_0^s 
d\gamma^k(I_N,Nx)$, respectively. Then we obtain 
for $L \in \mathbb N$  			
 \begin{align*} 
 			&  \left| G(s) - \frac{1}{\lv I_N\rv}\int_{0}^s  \,
 			d\sigma (I_N,Nx) \right| \nonumber  
 			\leq \sum_{k=2}^L \left|
 			\mathbb{E}_{N}\left(\int_0^s d\gamma^k(I_N,Nx) \right) -
 			\int_0^s d\gamma^k(I_N,Nx)\right| \nonumber 
 			\\ 
 			+& \sum_{k=2}^L
 			\left| \mathbb{E}_{N} \left( \int_0^s d\gamma^k(I_N,Nx) \right)
 			- E(s,k) \right| + E(s,L)+E(s,L+1). 
 		\end{align*}
Introducing the notation $\bar{s}:=\max(1,s)$, we conclude from \eqref{int_chi}
\begin{align*}
 E(s,L)+ E(s,L+1)\leq \frac{C_0^{L+1} \bar{s}^L}{(L-1)!}
\end{align*}
and using (1), we arrive at
\begin{align*}
 & \left| G(s) - \frac{1}{\lv I_N\rv}\int_{0}^s  \,
 			d\sigma (I_N,Nx) \right| 
 			 \leq \sum_{k=2}^L \left|
 			\mathbb{E}_{N}\left( \int_{0}^s  \, d\gamma^k(I_N,Nx) \right) -
 			\int_{0}^s  \, d\gamma^k(I_N,Nx) \right|
 			\\ 
 			&\quad + \frac{C_0^{L+1} \bar{s}^L}{(L-1)!}+\mathcal{O}\left(\frac{1}{|I_N|}\right)  \sum_{k=2}^L   
\frac{s^k C_0^k}{(k-1)!}  
 			+\frac{1}{\k_N}\sum_{k=2}^L s^{k-1}   C_0^k \frac{k^{k/2+1}}{(k-1)!} .  
\end{align*}
Now, under our assumptions on the growth of $s$ and $L$, $\frac{C_0^{L+1} \bar{s}^L}{(L-1)!}$ is $o(1/A_N)$. The sum
\begin{align*}
\sum_{k=2}^L\frac{s^k C_0^k}{(k-1)!} =sC_0\sum_{k=1}^{L-1}   
\frac{s^k C_0^k}{k!}  
\end{align*}
is $\O(A_N^\e)$ for any $\e>0$ by Lemma \ref{lemma_sums}, applied with $f(z)=e^{C_0z}$. To deal with the remaining sum, first 
observe 
that the series 
\begin{align*}
f(z):=\sum_{k=0}^\infty s^{k}   C_0^k \frac{k^{k/2+2}}{k!}
\end{align*}
converges absolutely and defines an entire function. Using \eqref{order} and Stirling's formula, we readily compute its order 
as $\rho=2$ and using \eqref{type} and again Stirling's formula, its type as $\nu=C_0^2e/2$. Thus Lemma \ref{lemma_sums} finishes 
the proof.
\end{proof}

%
%
%

\subsection{The variance of $\displaystyle \int_0^s d\gamma^k(I_N,Nx) $} \label{sec_variance}
Taking expectations in Lemma \ref{pointwise} (2), we arrive at a sum of expected absolute differences of $\int_0^s 
d\gamma^k(I_N,Nx)$ to its expectation. We will estimate this in terms of  the squareroot of the variance of $\int_0^s 
d\gamma^k(I_N,Nx)$, which we now bound.

\begin{lemma}\label{Lemma_VAR}
In the situation of Lemma \ref{pointwise}, there exists a positive constant $C$ such that for $k \in \mathbb N, k \geq 2$ and 
$N$ sufficiently large we have
\begin{align*}
&\textup{Var} \left( \int_0^s d\gamma^k(I_N,Nx) \right) 
\leq \frac{C(2\bar{s})^{2k}C_0^{2k}  2^k}{|I_N| (k-1)!}   +\frac{(2k)^{k+1}}{(k!)^2} C_0^{3k} k^2 \mathcal{O}\left( 
\frac{1}{\k_N}\right) \bar{s}^{2k}
\end{align*}
\end{lemma}
\begin{proof}
In order to calculate the variance of $\int_0^s d\gamma^k(I_N,Nx)$, we first consider the second 
moment 
$\mathbb E_N\left( \left( \int_0^s d\gamma^k(I_N,Nx)\right)^2 \right)$ using the representation  (\ref{Darst:gamma_N}).
We expand  
\begin{align}
&\mathbb E_N\left( \left( \int_0^s d\gamma^k(I_N,Nx)\right)^2 \right)\nonumber\\
&= \frac{1}{|I_N|^2} \sum_{l=k}^{2k} \sum_{T,M \subset \{1,\ldots,N\} \atop |T|=|M|=k, T\cup M =l} \int_{\mathbb R^N}  
\chi_{s,I_N}(N t_T) \chi_{s,I_N}(N t_M) R_{N}^N (t) dt.
\label{empty_intersec}
\end{align}
First, we consider the inner sum in (\ref{empty_intersec}) for $l=2k$, i.e.~in the case that $T,M \subset \{1,\ldots,N\}$ satisfy  
$|T|=|M|=k$ and $T\cap M = \emptyset$.  Since there are $\binom{N}{k} \binom{N-k}{k}=\frac{N! }{k!^2 (N-2k)!}$ such sets,  we 
obtain 
by the symmetry of the correlation functions
\begin{align}
 & \frac{1}{|I_N|^2}  \sum_{T,M \subset \{1,\ldots,N\} \atop |T|=|M|=k, T\cap M = \emptyset} \int_{\mathbb R^N}  
\chi_{s,I_N}(N t_T) \chi_{s,I_N}(N t_M) R_{N}^N (t) dt\nonumber
 \\
 =&  \frac{1}{(k!)^2} \frac{1}{|I_N|^2 N^{2k}}\int_{\mathbb R^{2k}}  \chi_{s,I_N}(t_1,\ldots,t_k) \chi_{s,I_N}(t_{k+1},\ldots,t_{2k}) 
R_{N}^{2k} 
(t_1/N,\ldots,t_{2k}/N) dt_1 \ldots dt_k.
 \label{second_moment}
\end{align}
We consider the terms with $l< 2k$   in (\ref{empty_intersec}) later and observe that 
 by (\ref{Expect_gamma})  with $t'\coloneqq (t_1,\ldots,t_k), t'' \coloneqq (t_{k+1}, \ldots, t_{2k})$ we have
\begin{align}
 &\left(\mathbb  E\left(   \int_0^s d\gamma^k(I_N,Nx) \right)\right)^2\nonumber\\
&=\frac{1}{(k!)^2} \frac{1}{|I_N|^2 N^{2k}}\int_{\mathbb R^{2k}}  
\chi_{s,I_N}(t') \chi_{s,I_N}(t'') R_{N}^{k} (t'/N)R_N^{k} (t''/N) dt_1 \ldots dt_k.\label{second_moment_2}
\end{align}
To calculate  the difference of (\ref{second_moment}) and (\ref{second_moment_2}) we use
\begin{align*}
 &\frac{1}{N^{2k}}(R_{N}^{2k} (t'/N,t''/N) - R_N^k (t'/N) R_N^k(t''/N))\\
 &=  S_{2k} (t',t'') - S_{k} (t') S_{k}(t'') +  (2k)^{k+1} C_0^{3k}  \mathcal{O}\left( \frac{1}{\k_N}\right).
\end{align*}
By (\ref{int_chi})
we obtain 
\begin{align*}
& \frac{1}{(k!)^2} \frac{1}{|I_N|^2 N^{2k}}\int_{\mathbb R^{2k}}  \chi_{s,I_N}(t') \chi_{s,I_N}(t'')( R_{N}^{2k} (t'/N,t''/N) - R_N^k 
(t'/N) 
R_N^k(t''/N)) dt'  dt'' \\
&= \frac{1}{(k!)^2} \frac{1}{|I_N|^2}\int_{\mathbb R^{2k}}  \chi_{s,I_N}(t') \chi_{s,I_N}(t'')( S_{2k} (t',t'') - S_{k} (t') 
S_{k}(t'') ) dt'  
dt'' \\
&\quad \quad  +   \frac{(2k)^{k+1}}{(k!)^2} C_0^{3k}  k^2 \mathcal{O}\left( \frac{1}{\k_N}\right) \bar{s}^{2k}. 
\end{align*}
We now claim that 
\begin{align}
\label{est_S}
S_{2k} (t',t'') - S_{k} (t') S_{k}(t'') \leq 0, \quad t',t'' \in \mathbb R^{k}.
\end{align}
If two components of $(t',t'')$ are equal, then $S_{2k} (t',t'')=0$ and the claim is trivially true, as $S_k\geq0$. If all 
components are distinct, then  
\begin{align*}
 \left( \frac{\sin(\pi(t_n-t_m))}{\pi(t_n-t_m)}\right)_{1 
\leq n,m \leq j}
\end{align*} 
is a positive-definite 
matrix (its principal minors are exactly $S_1,S_2,\ldots,S_{j-1}> 0$). Now, \eqref{est_S} follows from Fischer's 
inequality.
With (\ref{est_S}) we can further estimate
\begin{align*}
&\frac{1}{(k!)^2} \frac{1}{|I_N|^2 N^{2k}}\int_{\mathbb R^{2k}}  \chi_{s,I_N}(t') \chi_{s,I_N}(t'')( R_{N}^{2k} (t'/N,t''/N) - R_N^k 
(t'/N) 
R_N^k(t''/N)) dt'  dt''
\\&\leq \frac{(2k)^{k+1}}{(k!)^2} C_0^{3k} k^2 \mathcal{O}\left( \frac{1}{\k_N}\right) \bar{s}^{2k}.
\end{align*}
So far, we showed
\begin{align*}
&\textup{Var}\left(\int_0^s d\gamma^k(I_N,Nx) \right)
\leq \frac{1}{|I_N|^2} \sum_{l=k}^{2k-1} \sum_{T,M \subset \{1,\ldots,N\} \atop |T|=|M|=k, T\cup M =l} \int_{\mathbb 
R^N}  \chi_{s,I_N}(Nt_T) \chi_{s,I_N}(N t_M) R_{N}^N(t) dt 
\\
&+\frac{(2k)^{k+1}}{(k!)^2} C_0^{3k}  k^2
\mathcal{O}\left( \frac{1}{\k_N}\right) \bar{s}^{2k}.
\end{align*}

We continue to consider the integrals in the double sum.
For $l\in \{k,\ldots,2k-1\}$ and sets $T,M \subset \{1,\ldots,N\}$ with   $|T|=|M|=k$ and $T\cup M =l$ (i.e. $T$ and $M$ have a 
non-empty intersection) we have (using the symmetry of $R_N^k$ and $\frac{1}{N^l}R_{N}^l \leq  C_0^l$)
\begin{align}
& \int_{\mathbb R^N}  \chi_{s,I_N}(N t_T) \chi_{s,I_N}(N t_M) R_{N}^N (t) dt \nonumber \\
  \leq & \frac{(N-l)!}{N!} \frac{1}{N^l} \int_{I_N^l} 
 \dopp{1}_{(0,2s)} \left(\max_{i=1,\ldots,l} t_i 	-\min_{i=1,\ldots,l} t_i \right) R_{N}^l(t_1/N,\ldots,t_l/N) dt_1 \ldots 
dt_l  
\nonumber \\
 \leq & \frac{(N-l)!}{N!} l|I_N| (2\bar{s})^{2k}  C_0^{2k}. \label{int_chi_2}
\end{align}
We observe that for given $l$ there are $C_{N,l,k} \coloneqq\binom{N}{l} \binom{l}{k} \binom{k}{2k-l}$ sets $T,M \subset 
\{1,\ldots,N\}$ with   $|T|=|M|=k$ and $T\cup M =l$ by some easy combinatorial argument. 
Hence, by (\ref{int_chi_2}) and  $\frac{1}{N!} (N-l)! C_{N,l,k}= \frac{1}{(2k-l)! (l-k)!^2}$, we obtain 
\begin{align*}
&  \frac{1}{|I_N|^2} \sum_{l=k}^{2k-1} \sum_{T,M \subset \{1,\ldots,N\} \atop |T|=|M|=k, T\cup M =l} \int_{\mathbb R^N}
\chi_{s,I_N}(t_T) \chi_{s,I_N}(t_M) R_{N,N} (t/N) dt  \\
  & \leq
 \frac{1}{|I_N| }(2\bar{s})^{2k}C_0^{2k}  \sum_{l=k}^{2k-1} \frac{l}{(2k-l)! (l-k)!^2}.
\end{align*}
By the easy calculation 
\begin{align*}
 \sum_{l=k}^{2k-1} \frac{l}{(2k-l)! (l-k)!^2} =  \frac{1}{k!}\sum_{l=0}^{k-1} \frac{l+k}{ l!} \binom{k}{l} \leq 
\frac{2k}{k!}\sum_{l=0}^{k-1} \binom{k}{l} \leq \frac{2^k}{(k-1)!},
\end{align*}
we get
\begin{align*}
   \frac{1}{|I_N|^2} \sum_{l=k}^{2k-1} \sum_{T,M \subset \{1,\ldots,N\} \atop |T|=|M|=k, T\cup M =l} \int_{\mathbb R^N}  
\chi_{s,I_N}(t_T) \chi_{s,I_N}(t_M) R_{N}^N (t/N) dt 
  & \leq
 \frac{(2\bar{s})^{2k}C_0^{2k} }{|I_N| }  \frac{2^k}{(k-1)!}.
\end{align*}
Summarizing, we have shown
\begin{align*}
&\textup{Var}\left(\int_0^s d\gamma^k(I_N,Nx) \right) \leq \frac{C}{|I_N| }(2\bar{s})^{2k}C_0^{2k}  2^k\frac{1}{(k-1)!} 
  +\frac{(2k)^{k+1}}{(k!)^2} C_0^{3k}  k^2 \mathcal{O}\left( \frac{1}{\k_N}\right) \bar{s}^{2k}.
\end{align*}

\end{proof}

>From Lemma \ref{Lemma_VAR}, we can already derive an estimate on the expected deviation of the spacing distribution from its limit 
at a given point $s$.
\begin{corollary} \label{koro_erw_2}
Let the conditions of Lemma \ref{pointwise} be satisfied.
Then we have for any $\e>0, s\geq 0$
\begin{align*}
\mathbb E_N\left( \left| \frac{1}{\lv I_N\rv}\int_0^{s} d\sigma(I_N, Nx) - G(s) \right| \right) =
\O\left(\frac{1}{A_N^{1/2-\e}}\right).
 \end{align*} 
\end{corollary}
\begin{proof}
 By statement (2) in Lemma \ref{pointwise}, it suffices to bound
\begin{align*}
\mathbb E_N \left( \sum_{k=2}^L \left| \mathbb E_N\left(  \int_0^{s} d\gamma^k(I_N,Nx)\right) -  \int_0^{s} 
d\gamma^k(I_N,Nx)) \right| \right) 
 \leq   \sum_{k=2}^L \sqrt{ \text{Var} \left( \int_0^{s} d\gamma^k(I_N,Nx) \right)}. 
\end{align*}

The subadditivity of the square root together with  Lemma \ref{Lemma_VAR} give
\begin{align*}
\sum_{k=2}^L \sqrt{ \text{Var} \left( \int_0^{s} d\gamma^k(I_N,Nx) \right)} 
&=  \O\left(\frac{1}{\sqrt{|I_N|}}\right)  \sum_{k=2}^L  \frac{1}{\sqrt{(k-1)!}}  \bar{s}^{k} (2C_0)^{k}   \sqrt{2}^k \\
&     
 +\O\left(\frac{1}{\sqrt{\kappa_N}}\right)\sum_{k=2}^L  \frac{(2k)^{(k+1)/2}}{k!} C_0^{3k/2}  k \bar{s}^{k}.
\end{align*}
Similar to the proof of Lemma \ref{pointwise}, we can apply Lemma \ref{lemma_sums} with the functions 
\begin{align*}
 &f_1(z):=\sum_{k=0}^\infty  \frac{1}{\sqrt{k!}}  z^{k} (2C_0)^{k}\sqrt{2}^k,\\
&f_2(z):=\sum_{k=0}^\infty  \frac{(2k)^{(k+1)/2}}{k!} C_0^{k} C_0^{k/2}  k z^{k},
\end{align*}
which are both of order 2 and finite type, as can be  checked easily using \eqref{order} and \eqref{type}.
\end{proof}

\subsection{Completing the proof of Theorem \ref{thrmgeneral}}\label{sec_sup}
 The idea for the rest of the proof of Theorem \ref{thrmgeneral} is to replace the supremum over $s$ with a maximum over a finite 
set of certain nodes $s_i$.  Then, we can choose the number of these nodes, growing with $N$ in such a way, that the error estimates 
lead to the error claimed in Theorem \ref{thrmgeneral}.

 Let $M=M(N)\in\mathbb 
N$ and let $s_i$, $i=0,\dots,M$, denote the $i/M$-quantile of $G$, that is
\begin{align*}
 G(s_i)=\frac{i}{M}, \quad i=0,\ldots,M.
\end{align*}
The existence of these nodes is ensured as $G$ is continuous by definition, increasing and $\lim_{t\to\infty}G(t)=1$. 
Now, set 
\begin{align*}
   \Delta(s,I_N, Nx)\coloneqq \frac{1}{\lv I_N\rv} \int_0^s d \sigma(I_N,Nx) - G(s),
\quad 
\Delta_M(I_N, Nx)\coloneqq \max_{i=1,\ldots, M-1} \lv\Delta(s_i,I_N,Nx)\rv.
\end{align*}
Furthermore, denote the largest non-trivial node by 
\begin{align*}
 \delta_N \coloneqq \max(1,s_{M-1}).
\end{align*}
It is known (cf.~\cite[Proposition 3.1.9]{KatzSarnak}) that the Gaudin distribution has sub-Gaussian tails, that is,  
\begin{align}
\label{tail_est}
 1-G(s)\leq Ae^{-Bs^2}
\end{align}
for some $A,B>0$. This implies that $\d_N$ fulfils
\begin{align}
 \d_N=\O(\sqrt{\log M}).\label{deltabound}
\end{align}

\begin{proof}[Proof of Theorem \ref{thrmgeneral}]
We first establish the inequality
 \begin{align}\label{KS1} 
\mathbb E_N \left( \sup_{s \in \mathbb R}|\Delta(s,I_N,Nx)| \right)\leq \frac{1}{M} + \mathbb E_N(\Delta_M(I_N, Nx)) 
+ \mathbb E_N\left( \left|\frac{1}{\lv I_N\rv}\int_{\mathbb R} d\sigma (I_N,Nx)-1 \right|\right).
\end{align}
It has been given in \cite{KatzSarnak}, so we only sketch its simple proof.
For fixed $s_{j-1}\leq s\leq s_j$ with $j\leq M-1$, we have 
\begin{align*}
 &\Delta(s,I_N,Nx)\leq \lv I_N\rv^{-1}\int_0^{s_j} 
d\s(I_N,Nx)-G(s_{j-1})=\Delta(s_j,I_N,Nx)+G(s_j)-G(s_{j-1})\\
&=\Delta(s_j,I_N,Nx)+\frac{1}{M}\leq \Delta_M(I_N,Nx)+\frac{1}{M}.
\end{align*}
Similarly,
\begin{align*}
 -\Delta(s,I_N,Nx)\leq\Delta_M(I_N,Nx)+\frac{1}{M}.
\end{align*}
The term $\left|\frac{1}{\lv I_N\rv}\int_{\mathbb R} d\sigma (I_N,Nx)-1 \right|$ stems from an analogous estimate for 
$\Delta(s_j,I_N,Nx)$ with $s>s_{M-1}$.

Next, we show that 
\begin{align}
\label{total_mass_sigma}
  \mathbb E_N \left(\left| \frac{1}{\lv I_N\rv}\int_{\mathbb R} d\sigma(I_N,Nx) -1\right| \right) =\O(A_N^{-\frac12}).
  \end{align}
Using the notation $S(I_N,Nx) := \# \{i: Nx_i \in I_N\} $,  we can write
\begin{align}
 \mathbb E_N \left|\frac{1}{\lv I_N\rv}\int_{\mathbb R} d\sigma (I_N,Nx)-1 \right|^2 &= \mathbb E_N  
\left(\frac{S(I_N,Nx) 
-1}{|I_N|} \right)^2   - 2\mathbb E_N \frac{S(I_N,Nx) -1}{|I_N|} +1.\label{bound_for_Markov}
\end{align}
Hence, we need to calculate the first and second moment of $S(I_N,Nx)$.
By an easy computation, we obtain   (using the symmetry of $R_{N}^N$)
\begin{align}
\mathbb E_N(S(I_N, Nx)) &= \int_{t_1, \ldots, t_N} \left( \sum_{i=1}^N \dopp{1}_{I_N} (N t_i) \right) R_{N}^N 
(t_1,\ldots,t_N) dt_1 \ldots dt_N
\nonumber\\
 &=\frac{1}{N}\int_{I_N}       R_{N}^1 
(t_1/N) dt_1 
 = |I_N| \left(1+\mathcal{O}\left( \frac{1}{\k_N}\right)\right).\label{firstmoment}
\end{align}  
In a similar fashion, using $\int_{I_N^2} S_2(x,y)dxdy = |I_N|^2 + \mathcal{O}(|I_N|)$, we have  
 \begin{align*}
\mathbb E_N(S(I_N,Nx)^2) 
&= \int_{t_1,\ldots, \leq t_N} \left(\sum_{i,j=1}^N \dopp{1}_{I_N} (Nt_i)\dopp{1}_{I_N} (N t_j)\right)  R_{N}^N 
(t_1,\ldots,t_N) dt_1 \ldots dt_N 
\\ & = |I_N|^2 \left(1+\mathcal{O}\left(\frac{1}{\kappa_N}\right)\right) + \mathcal{O}(|I_N|),
\end{align*}    
which shows together with \eqref{bound_for_Markov} and \eqref{firstmoment}
\begin{align}\label{for_Markov}
 \E_N \left|\frac{1}{\lv I_N\rv}\int_{\mathbb R} d\sigma (I_N,Nx)-1 \right|^2=\O(A_N^{-1}).
\end{align}
Now Jensen's inequality proves the claim in (\ref{total_mass_sigma}). 
We further use the crude bound
\begin{align*}
\mathbb E_N(\Delta_M (I_N, Nx)) 
\leq \sum_{i=1}^{M-1} \mathbb E_N \left( |\Delta(s_i,I_N, Nx)|  \right).
\end{align*}
Now, choosing $M$ as the smallest natural number larger than $A_N^{1/4}$, we get with \eqref{KS1} and Corollary  
\ref{koro_erw_2}
\begin{align*}
 \mathbb E_N \left( \sup_{s \in \mathbb R}|\Delta(s,I_N,Nx)| \right)\leq 
\frac1{A_N^{1/4}}+A_N^{1/4}\O\left(\frac{1}{A_N^{1/2-\e}}\right)+\O\left(\frac1{A_N^{1/2}}\right),
\end{align*}
which proves Theorem \ref{thrmgeneral} for $\lv I_N\rv^{-1}\s(I_N,Nx)$. 
To deduce the result for $\hat{\s}(I_N,Nx)$, let for $0<\iota<1$ denote 
\begin{align*}
 A:=\left\{x\,:\, \left\lv \frac{\int_0^\infty d\s(I_N,Nx)}{\lv I_N\rv}-1\right\rv\leq  A_N^{-\iota}\right\}.
\end{align*}
We will assume that $N$ is so large that $x\in A$ implies $\int_0^\infty d\s(I_N,Nx)>0$.
Then 
\begin{align*}
 &\E_N\left( \sup_{s \in \mathbb R} \left| \int_0^s  d\hat{\s}(I_N,Nx) - G(s) 
\right|\right)\\
&\leq \E_N\left(\dopp{1}_A(x) \sup_{s \in \mathbb R} \left|\frac{\lv I_N\rv}{\int_0^\infty d\s(I_N,Nx)} \int_0^s  \frac1{\lv 
I_N\rv}d\s(I_N,Nx) - G(s) 
\right|\right)+\mathbb P_N(A^c).
\end{align*}
It is straightforward to check that $x\in A$ implies $\lv I_N\rv/\int_0^\infty d\s(I_N,Nx)=1+\O(A_N^{-\iota})$, where the 
$\O$ term is independent of the specific $x$. This gives
\begin{align*}
 &\E_N\left( \sup_{s \in \mathbb R} \left| \int_0^s  d\hat{\s}(I_N,Nx) - G(s) 
\right|\right)\\
&\leq \E_N\left(\dopp{1}_A(x) \sup_{s \in \mathbb R} \left|\int_0^s  \frac1{\lv 
I_N\rv}d\s(I_N,Nx) - G(s) 
\right|+\O(A_N^{-\iota})\dopp{1}_A(x)\sup_{s \in \mathbb R}   \frac{\int_0^sd\s(I_N,Nx)}{\lv 
I_N\rv}\right)\\&+\mathbb P_N(A^c)\\
&=\O(A_N^{-1/4+\e})+\O(A_N^{-\iota})+\mathbb P_N(A^c).
\end{align*}
It remains to estimate the probability of $A^c$. The bound \eqref{for_Markov} gives with Markov's inequality 
\begin{align*}
\mathbb  P_N(A^c)=\O(A_N^{2\iota-1}).
\end{align*}
Now the theorem follows choosing $\iota=1/4$.
\end{proof}

\section{Proof of Theorem \ref{universality} (1) and (2)}\label{sec:proof_universality}
We need to introduce some of the notation of \cite{KSSV}. Let us define the Mhaskar-Rakhmanov-Saff numbers 
$a_V$ and $b_V$ via the relations
\begin{align}
 \int_{a_V}^{b_V}\frac{V'(t)}{\sqrt{(b_V-t)(t-a_V)}}dt=0,\quad 
\int_{a_V}^{b_V}\frac{tV'(t)}{\sqrt{(b_V-t)(t-a_V)}}dt=2\pi.\label{MRS numbers}
\end{align}
It is known that for convex, smooth $V$, $a_V$ and $b_V$ are uniquely determined by \eqref{MRS numbers} and that
these
are the endpoints of the support of the equilibrium measure $\mu_V$. Moreover, it is important for us to see them as functions of 
$V$.

The linear rescaling that maps $[-1,1]$ onto $[a_V,b_V]$ is denoted by
\begin{align}
 \map{\l_V}{\R}{\R},\  \l_V(s):=\frac{b_V-a_V}{2}s+\frac{b_V+a_V}{2}.\label{def_lambda}
\end{align}
Its inverse is
\begin{align}
 \l_V^{-1}(t)=\frac{2}{b_V-a_V}t-\frac{b_V+a_V}{b_V-a_V}.\label{inversescaling}
\end{align}

 Hence, $[-1,1]\subset \hat{J}:=\l_V^{-1}(J)$ (cf.~\GAi (1)). Moreover, set
\begin{align}
 &\map{h_V}{\hat{J}\times\hat{J}}{\R},\
h_V(t,x):=\int_0^1(V\circ\l_V)''(x+u(t-x))\, \mathrm du\label{def_h}\\
&\hspace{4.9cm}=\frac{(V\circ\l_V)'(t)-(V\circ\l_V)'(x)}{t-x},\nonumber\\
&\map{G_V}{\hat{J}}{\R},\ G_V(x):=\frac{1}{\pi}\int_{-1}^1\frac{h_V(t,x)}{\sqrt{1-t^2}}\, \mathrm dt,\label{def_G}\\
&\map{\rho_V}{\R}{\R},\ \rho_V(x):=\begin{cases}
                                   \frac{1}{2\pi}\sqrt{1-x^2}\, G_V(x)	&, \text{ if }\lv x\rv\leq 1 ,\\
				   0					&, \text{ else,} 
                                  \end{cases}\label{def_rho}\\
&\map{\hat{a}}{(-1,1)}{\R},\ \hat{a}(x):=\lb\frac{1-x}{1+x}\rb^{1/4}.\label{def_hat_a}
\end{align}
Note that $\rho_V$ is the equilibrium measure of $V$ rescaled such that its support is $[-1,1]$. The actual equilibrium measure 
$\mu_V$ is related to $\rho_V$ via
\begin{align*}
 \mu_V(t)=\frac{2}{b_V-a_V}\rho_V(\l_V^{-1}(t)).
\end{align*}

\begin{proof}[Proof of Theorem \ref{universality} (1) and (2)]
We will mostly deal with the more involved case (2).
Let us  abbreviate $V,f$ for $V-f/N$. We will also write $\|\cdot\|$ instead of $\|\cdot\|_D$.

Combining \cite[Theorem 1.3]{KSSV} with \cite[Proposition 4.1]{KSSV} gives
\begin{align} 
&\frac{b_{V,f}-a_{V,f}}{2}K_{N,V,f}(\l_{V,f}(r),\l_{V,f}(s))=\frac{1}{2\pi}\lb\frac{\hat{a}(r)}{\hat{a}(s)}+\frac{\hat{a}(s)}{\hat{a
} (r)}\rb \frac{\sin\lb N\pi\int_s^r\rho_{V,f}(s)\, \mathrm ds\rb}{r-s}\nonumber\\
& + \frac{1}{2\pi} \cos\lb\frac{N}{2}g(r,s)\rb\lb\frac{1}{\hat{a}(r)}+\frac{1}{\hat{a}(s)}\rb 
\frac{\hat{a}(r)-\hat{a}(s)}{r-s}+\O\lb\frac1N\rb,\label{bulk-formula}
\end{align}
where $g(r,s)$ is some function which is not important here.
Formula \eqref{bulk-formula} holds  
for all $r,s\in (-1+\d,1-\d)$ with arbitrary $\d>0$, where the $\O$ term is uniform in $r,s$ for fixed $\d$  and uniform for 
$V-f/N\in X_D$ for some $D$. 
In order to use \eqref{bulk-formula} for the proof of Theorem \ref{universality}, we first  have to show that for $N$ large enough and some $\d>0$
\begin{align}
 \l_{V,f}^{-1}(F_V^{-1}(t/N))\in (-1+\d,1-\d), \text{ for all }f\text{ with }\|f\|\leq 
N^\eta\label{asympsupport}
\end{align}
for some $\eta>0$ small.
By 
\cite[Lemma 2.4]{KSSV}, the maps $V\mapsto a_V, V\mapsto b_V$, defined by \eqref{MRS numbers} are Frechet differentiable with 
(uniformly) bounded derivatives on a neighborhood of $V$. This lemma was proved in \cite{KSSV} only for $V$ satisfying \GAii\ but 
the proof goes through also for $V$ with \GAi. Hence
\begin{align}
 a_{V,f}=a_V(1+\O(\|f\|/N)),\ b_{V,f}=b_V(1+\O(\|f\|/N))\label{approx1}
\end{align}
and thus
\begin{align}
 \l_{V,f}(t)=\l_V(t)(1+\O(\|f\|/N))\label{approx2}
\end{align}
uniformly for $ t \in [-1,1]$.
For given $0<\eta<1$, assertions \eqref{approx1} and \eqref{approx2} hold uniformly for all $f$ with 
$\|f\|\leq N^\eta$, which proves \eqref{asympsupport}. We have thus shown that formula \eqref{bulk-formula} can be applied. 
The second summand on the rhs \eqref{bulk-formula} is uniformly bounded for $r,s\in (-1+\d,1-\d)$ and hence negligible when multiplied by $\frac{1}{N}$. 
Now, we claim that uniformly on $(-1+\d,1-\d)$
\begin{align*}
 \frac{\hat{a}(r)}{\hat{a}(s)}+\frac{\hat{a}(s)}{\hat{a
} (r)}=2+\O(\lv r-s\rv^2).
\end{align*}
Setting $z:=\frac{\hat{a}(r)}{\hat{a}(s)}$, this claim is equivalent to the relation
$(z-1)^2/z=\O(\lv r-s\rv^2)$. It is now a straightforward application of Taylor's formula to show $z-1=\O(\lv r-s\rv)$.
Writing
\begin{align}
 \hat{t}:= F_V^{-1}\left(\frac{t}{N}\right) , \quad \hat{s}:= F_V^{-1}\left(\frac{s}{N} \right),\label{abbrev}
\end{align}
we arrive at
\begin{align}
 \frac{1}{N\mu_V(\hat{t})}K_{N,V,f}(\hat{t}, \hat{s})=\frac{\sin\lb 
N\pi\int_{\l_{V,f}^{-1}(\hat{s})}^{\l_{V,f}^{-1}(\hat{t})}\rho_{V,f}(r)\, \mathrm 
dr\rb}{N\mu_V(\hat{t})\pi(\hat{t}-\hat{s})}+\O(1/N).\label{kernelinterm}
\end{align}
Now, using \eqref{approx1} together with the definitions \eqref{inversescaling}, \eqref{def_h}, \eqref{def_G} and \eqref{def_rho}, 
we find
\begin{align}
 \l_{V,f}(t)^{-1}=\l_V(t)^{-1}(1+\O(\|f\|/N))\ \text{ and }\ \rho_{V,f}(t)=\rho_V(t)(1+\O(\|f\|/N))\label{approxrho}
\end{align}
uniformly on $[-1,1]$. Define
\begin{align*}
 \mu_{V,f}(t):=\frac{2}{b_{V,f}-a_{V,f}}\rho_{V,f}(\l_{V,f}^{-1}(t))\ \text{ and }\ F_{V,f}(t):=\int_{a_{V,f}}^t\mu_{V,f}(s)ds.
\end{align*}
Furthermore, let $g_{V,f}(t):=F_{V,f}(t)-F_V(t)$. We conclude
\begin{align*}
&\int_{\l_{V,f}^{-1}(\hat{s})}^{\l_{V,f}^{-1}(\hat{t})}\rho_{V,f}(r)dr=(F_{V}+g_{V,f})(\hat{t})-(F_{V}+g_{V,f})(\hat{s})=\frac{t-s}{
N}+g_{V,f}(\hat{t})-g_{V,f}(\hat{s}).
\end{align*}
Using \eqref{approxrho} and the smoothness of $\rho_V$ (see \eqref{def_rho}), it is 
straightforward to establish the relation 
\begin{align*}
 \mu_{V,f}(t)=\mu_V(t)+\O(\|f\|/N)
\end{align*}
uniformly on $\R$. It follows that
\begin{align*}
 g_{V,f}(\hat{t})-g_{V,f}(\hat{s})=\O(\|f\|/N)\lv F_V^{-1}(t/N)-F_V^{-1}(s/N)\rv=\O(\|f\|/N)\frac{\lv t-s\rv}{N},
\end{align*}
 where we used in the last step that $\mu_V$ is bounded away from 0 on $[a_V+\d, b_V-\d]$. Hence \eqref{kernelinterm} reduces to
\begin{align}
 \frac{1}{N\mu_V(\hat{t})}K_{N,V,f}(\hat{t}, 
\hat{s})&=\frac{\sin\lbb\pi (t-s)+\O(\|f\|/N)\lv t-s\rv\rbb}{N\mu_V(\hat{t})\pi(\hat{t}-\hat{s})}+\O(1/N)\label{kernel1}\\
&=\frac{\sin\lbb\pi (t-s)\rbb}{N\mu_V(\hat{t})\pi(\hat{t}-\hat{s})}+\frac{\O(\|f\|)\lv 
t-s\rv}{N^2\mu_V(\hat{t})\pi(\hat{t}-\hat{s})}+\O(1/N)\label{kernel2},
\end{align}
where we used the simple inequality $\lv\sin(t+s)-\sin(t)\rv\leq \lv s\rv$.

If $1/\lvert t-s\rvert=\mathcal{O}(1/N)$, i.e.~$\frac{|t-s|}{N} \geq c$ for some constant $c>0$, then $F_V^{-1}(t/N)-
F_V^{-1}(s/N)$ is bounded away from 0 and hence we see that the first term on the right-hand side of \eqref{kernel1} is 
$\mathcal{O}(1/N)$ which
proves the theorem in this case (as the sine kernel of $t-s$ is then also $\mathcal{O}(1/N)$).

If $\lvert t-s\rvert=o(N),$ using Taylor's expansion on $ F_V^{-1}\left(\frac{s}{N}\right)$ at $\frac{t}{N}$ leads for 
some
$\nu$ between $t/N$ and $s/N$ to
\begin{align*}
 F_V^{-1}\left(\frac{t}{N}\right)- F_V^{-1}\left(\frac{s}{N}\right)
&=  \frac{1}{ \mu_V(F_V^{-1}(t/N)) } \frac{(t-s)}{N} -\frac{1}{ 2 } (F_V^{-1})^{\prime \prime} (\nu)\frac{(t-s)^2}{N^2}.
\end{align*}
Hence, the second summand in \eqref{kernel2} is $\O(\|f\|/N)$ and we arrive at
\begin{align}
\frac{1}{N \mu_V(\hat{t})} K_{N,V,f} \left( \hat{t},\hat{s}\right)= \frac{\sin(\pi (t-s)) }{\pi (t-s) - \pi
A(t,\nu,N)\frac{(t-s)^2}{N}}+\mathcal{O}(\|f\|/N)\label{estimatef},
\end{align}
where
\begin{align*}
 A(t,\nu,N):=\mu_V( F_V^{-1}\left(t/N\right)) \frac{1}{ 2 } (F_V^{-1})^{\prime \prime} (\nu).
\end{align*}
Note that in \eqref{estimatef} the $\O$-term will be of order $N^{-1}$ for $f$ fixed in case (1) and $N^{\eta-1}$ in case (2).
By the simple equality 
$$\frac{1}{a+b} - \frac{1}{a} =\frac{-b}{a (a+b)} $$
we get
\begin{align*}
 \frac{\sin(\pi (t-s)) }{\pi (t-s) -\pi A(t,\nu,N)\frac{(t-s)^2}{N}}
 &=\frac{\sin(\pi (t-s)) }{\pi (t-s) }+ \frac{\sin(\pi (t-s))\pi A(t,\nu,N)\frac{(t-s)^2}{N}}{ \pi (t-s)( \pi(t-s)-\pi
A(t,\nu,N)\frac{(t-s)^2}{N})}
\\
&=\frac{\sin(\pi (t-s)) }{\pi (t-s) }
+\frac{1}{\pi N} \frac{\sin(\pi (t-s))A(t,\nu,N)}{ 1 -  A(t,\nu,N)\frac{(t-s)}{N}}
\\
&=\frac{\sin(\pi (t-s)) }{\pi (t-s) }
+\mathcal{O}\left(\frac{1}{ N}\right).
\end{align*}
The last equality is due to the boundedness of $A(t,\nu,N)$ for $t \in I_N$ and our assumption $(t-s)/N\to0$.
\end{proof}

\section{Repulsive Particle Systems -- Proof of Theorem \ref{universality} (3)}\label{sec_rep_part}

Note that part (3) of Theorem \ref{universality} for unitary invariant ensembles follows immediately from part (1) and the 
determinantal relations \eqref{determinantal_relations}. To prove (3) for repulsive particle systems, we need to introduce some of 
the method developed in 
\cite{GoetzeVenker} to tackle these ensembles. We remark that in comparison to \cite{GoetzeVenker}, there are several new 
(technical) elements, in particular the truncation to $\|f\|_D\leq N^\k$ and the necessity to work with complex-valued processes. 
Furthermore, aiming at rates of convergence requires a separate investigation of the cases of negative-definite $h$ and of 
arbitrary 
$h$.

The first step is to decompose the additional interaction term
$\sum_{i<j}h(x_i-x_j)$ into a linear term and a bivariate term of lower order. This will be done by the Hoeffding decomposition 
w.r.t.~ 
a (so far arbitrary)
probability measure $\mu$ on $\R$. Setting $h_\mu(t):=\int h(t-s)d\mu(s)$ and for another measure $\nu$ on $\R$ 
$h_{\mu\nu}:=\int\int h(t-s)d\mu(t)d\nu(s)$, we may write
\begin{align}
 &\sum_{i<j}h(x_i-x_j)=\frac{1}{2}\sum_{i,j}h(x_i-x_j)-\frac{N}{2}h(0)\nonumber\\
=&N\sum_{j=1}^Nh_\mu(x_j)+\frac{1}{2}\sum_{i,j}\left[h(x_i-x_j)-Nh_\mu(x_i)-Nh_\mu(x_j)+h_{\mu\mu}\right]+C_N,\label{e5}
\end{align}
where we set $C_N:=-(N/2) h(0)-(N^2/2)h_{\mu\mu}$. The term $N\sum_{j=1}^Nh_\mu(x_j)$ is of the same shape as 
$N\sum_{j=1}^NQ(x_j)$, giving rise to the external field $V_\mu:=Q+h_\mu$. Our aim is to choose $\mu$ such that $P_{N,Q}^h$ 
and the unitary invariant ensemble
$P_{N,V_\mu}$ have the same equilibrium measure. To achieve this, the statistic
\begin{align*}
 \U_\mu(x):=-\frac{1}{2}\sum_{i,j}\left[h(x_i-x_j)-Nh_\mu(x_i)-Nh_\mu(x_j)+h_{\mu\mu}\right]
\end{align*}
should be concentrated under $P_{N,V_\mu}$. As $\U_\mu$ is a global statistic, we should have
\begin{align}
 \lim_{N\to\infty}\frac{1}{N^2}\E_{N,V_\mu}\U_\mu(x)=-\frac12(h_{\nu\nu}-2h_{\mu\nu}+h_{\mu\mu}),\label{expectation}
\end{align}
  where $\nu$ now is the equilibrium measure to $V_\mu$. As we will not divide by $N^2$ and thus will be at the scale of 
fluctuations, the rhs of
\eqref{expectation} should be 0. This leads us to the condition $\nu=\mu$, or in other terms, $\mu$ should be the equilibrium 
measure to $V_\mu$. This recursive problem was solved by
a fixed point argument in \cite[Lemma 3.1]{GoetzeVenker}, showing existence of a measure $\mu$ with the desired property. The 
uniqueness followed later by proving that any
fixed point is the limiting measure for $P_{N,Q}^h$. From now on let $\mu$ denote the fixed point, set $V:=V_\mu$ and $\U:=\U_\mu$.
The identity
\begin{align*}
 P_{N,Q}^h(x)=\frac{Z_{N,V}}{Z_{N,V,\U}}P_{N,V}(x)e^{\U(x)}
\end{align*}
with $Z_{N,V,\U}:=Z_{N,Q}^he^{C_N}$
 allows to carry many properties from $P_{N,V}$ over to $P_{N,Q}^h$. Concentration of $\U$ under $P_{N,V}$ was proved in 
\cite[Proposition 4.7]{GoetzeVenker} by showing that the ratio
$Z_{N,V}/Z_{N,V,\U}$ is bounded in $N$ and bounded away from $0$ provided that $\a_Q$ is large enough. More precisely, for any
$\l>0$, there is a constant $\a(\l)<\infty$ such that for some $0<C_1<C_2<\infty$ and all  $\a_Q\geq \a(\l)$ we have
\begin{align}\label{concentration_U}
 C_1\leq\E_{N,V}e^{\l\U}\leq C_2
\end{align}
for all $N$.
A main ingredient to this is the following concentration of measure result for linear statistics (cf. \cite[Corollary 
4.4]{GoetzeVenker}), which will be used lateron.

\begin{proposition}\label{Concentration}
Let $Q$ be a real analytic external field with $Q''\geq c>0$. Then
for any Lipschitz function $f$ with third derivative bounded on an open interval $D$ containing $\supp\mu_Q$, we have for arbitrary
$\epsilon>0$
\begin{align*}
\E_{N,Q}\exp\lee{\epsilon\lbb\sum_{j=1}^N f(x_j)-N\int f(t)d\mu_Q(t)\rbb}\ree\leq \exp\lee{\frac{\epsilon^2\Lip{f}^2}{2c}}+
\e C (\|f\|_\infty+\|f^{(3)}\|_\infty)\ree.
\end{align*}
Here $C$ is uniform in $f$, $\Lip{f}$ denotes the Lipschitz constant of $f$ on $D$ and $\|\cdot\|_\infty$ is the sup norm 
on $D$.
\end{proposition}

The key to the
local statistics is a linearization method, which transforms the bivariate statistic $\U$ into random linear statistics. We
give an outline for negative-definite $h$, that means $\hat{h}\leq0$, where $\F{h}(t):=\frac{1}{\sqrt{2\pi}}\int_{\R}e^{-its}h(s)ds$ denotes the Fourier transform of $h$. For such 
function, $-h$ may be seen as the
covariance function of a centered stationary Gaussian process $(f(t))_{t\in\R}$, i.e.~a stochastic process on $\R$ whose 
finite-dimensional distributions are all multivariate Gaussian and such that $\textup{Cov}(f(t),f(s))=-h(t-s)$. Then a quick 
computation verifies that
\begin{align}
 \exp\{-\frac{1}{2}\sum_{i,j}h(x_i-x_j)\}=\E \exp\{\sum_{j=1}^Nf(x_j)\},\label{e6}
\end{align}
where the expectation is w.r.t. the probability space underlying the Gaussian process.  S. Jansen has pointed out to the second 
author that the linearization \eqref{e6} is known in mathematical physics as the 
Sine-Gordon transformation. Furthermore,
\begin{align}
 \exp\{\U(x)\}=\E \exp\{\sum_{j=1}^Nf(x_j)-N\int fd\mu\}\label{e7}
\end{align}
holds. The term $\sum_{j=1}^Nf(x_j)-N\int fd\mu$ can now be added to  $N\sum_{j=1}^NV(x_j)$, resulting in a perturbation of lower 
order, which does not influence the equilibrium measure. The limiting bulk correlations are not altered by the function
$f$ either, as can be seen from Theorem \ref{universality}.
It should be noted that the scaling of the correlation functions is independent of
$f$. 
To summarize, the ensemble $P_{N,Q}^h$ is an average over determinantal ensembles $P_{N,V-f/N}$ with a small random perturbation 
of the external field. We will show that universality of $P_{N,Q}^h$ can be deduced from universality of the
invariant ensembles $P_{N,V-f/N}$. Note that the averaging over $f$ results in a weaker rate of convergence as uniformity in $f$ 
has to be shown (cf. Theorem \ref{universality} (1) and (2)).

\subsection{Alternative representation of correlation functions and truncation}\noindent

Let us define the generalized invariant
ensemble
\begin{align}
 P_{N,Q,f}^M(x):=\frac{1}{Z_{N,Q,f}^M}\prod_{1\leq i<j\leq N}\lv x_i-x_j\rv^2e^{-M\sum_{j=1}^N Q(x_j)+\sum_{j=1}^N f(x_j)},
\end{align}
where $M\in\mathbb{N}$. If $M=N$, we have $P_{N,Q,f}= P_{N,Q,f}^M$ and $ 
P_{N,Q}^M= P_{N,Q,f}^M$, if
$f=0$. 
Then the $k$-th correlation function of $P_{N,V}$ at  $t_1,\dots,t_k$ can be rewritten as
\begin{align*}
&R^k_{N,V}(t_1,\dots,t_k)\nonumber\\
 =&\frac{N!}{(N-k)!}\int_{\R^{N-k}}\frac{1}{Z_{N,V}}\exp\lee-N\sum_{j=k+1}^NV(x_j)+2\sum\limits_{i<j;\ i,j>k} \log\lvb x_j-x_i\rvb
\ree\nonumber\\
&\times \exp\lee-N\sum_{j=1}^k
V(t_1,\dots,t_k)+2\sum_{i<j;\ i,j\leq
k}\log\lvb t_i-t_j\rvb\ree\nonumber\\
&\times\exp\lee 2\sum_{i\leq k,\ j>k}\log\lvb t_i-x_j\rvb\ree dx_{k+1}\dots dx_N\nonumber\\
&=\frac{N!}{(N-k)!}F(t)\frac{Z_{N-k,V}^N}{Z_{N,V}}\E_{N-k,V}^N\exp\lee{2\sum_{i\leq k,\ j>k}\log\lv t_i-x_j\rv}\ree\quad 
\text{with}\nonumber\\
&F(t):=\exp\lee{-N\sum_{j=1}^kV(t_j)+2\sum_{i<j;\ i,j\leq
k}\log\lv t_i-t_j\rv}\ree
\end{align*}
Labeling the eigenvalues of the ensemble $P_{N-k,V}^N$ by $x_{k+1},\dots,x_N$ and denoting
\begin{align}
 O:=O_{N-k,V}^N(t,x):=2\sum_{i\leq k,\ j>k}\log\lv
t_i-x_j\rv+\log\big[F(t)\frac{Z_{N-k,V}^N}{Z_{N,V}}\big],
\end{align}
we arrive at the representation
\begin{align}
&R^k_{N,V}(t_1,\dots,t_k)=\frac{N!}{(N-k)!}\E_{N-k,V}^N\exp\lee{O}\ree.
\end{align}
By analogous steps, we represent the $k$-th correlation function $R_{N,Q}^{h,k}$ of $P_{N,Q}^h$  as 
\begin{align}
R_{N,Q}^{h,k}(t_1\dots,t_k)=\frac{N!/(N-k)!}{\E_{N,V}\exp\lee{\U(x)}\ree}\E_{N-k,V}^N\exp\lee{\U(t,x)+O}\ree,\label{e8}
\end{align}
where we abbreviated $\U(t_1,\dots,t_k,x_{k+1},\dots,x_N)$ by $\U(t,x)$. By \cite[Lemma 28]{GoetzeVenker} we can
assume that $x_{k+1},\dots,x_N\in [-L,L]$ for $L$ large enough. To be precise, the lemma shows that for each $k$ we have $L,C>0$
such that
for all $N$ and
for all $t_1,\dots,t_k$
\begin{align}
 \lv R_{N,Q}^{h,k}(t_1,\dots,t_k)-\frac{N!/(N-k)!}{\E_{N,V;L}\exp\lee{\U(x)}\ree}\E_{N-k,V;L}^N\exp\lee{\U(t,x)+O_L}
\ree\rv\leq
e^{-CN},\label{truncation}
\end{align}
where $\E_{N,V;L}^M$ denotes expectation w.r.t. $P_{N,V;L}^M$, which is the normalized restriction of $P_{N,V}^M$
 to $[-L,L]^N$ and $O_L$ is the analog of $O$, obtained by replacing integrations over $\R$ by integrations 
over
$[-L,L]$. For later use, let us also state one more inequality from that lemma,
\begin{align}
 N^{-1}R_{N,Q}^{h,1}(t)\leq \exp\{{CN}-c_1N[V(t)-c_2\log(1+t^2)]\}\label{truncation_Q},
\end{align}
valid for some constants $C,c_1,c_2>0$.
It will allow us to restrict the whole ensemble (instead of correlation functions) to some compact $[-L,L]$.

\subsection{Linearization and proof of Theorem \ref{universality} (3) for negative-definite $h$}\noindent

Let us give more details on the linearization method for negative-definite $h$. For such $h$, $-h$ can
indeed be seen as the covariance function of a centered stationary Gaussian process on $\R$ such that \eqref{e6} and \eqref{e7} hold. 
Since the sample paths of that process will become a part of the external field, we have to show analyticity. This can be done by 
invoking an explicit representation. Recall that 
$\F{h}(t)$ denotes the Fourier transform of $h$ and that we have
$-\hat{h}\leq0$.

For the representation of $f$, let $(B^1_t)_t, (B^2_t)_t$ denote two independent Brownian motions and define
\begin{align}
 f(t):=(2/\pi)^{1/4}\int_0^\infty \cos(ts)\sqrt{-{\hat{h}(s)}}dB_s^1+(2/\pi)^{1/4}\int_0^\infty
\sin(ts)\sqrt{-{\hat{h}(s)}}dB_s^2.\label{process_representation}
\end{align}
Here it is convenient to understand the stochastic integral  as a Wiener integral,
\begin{align*}
 \int_0^\infty g(s)dB^1_s:=-\int_0^\infty B^1_sdg(s)=-\int_0^\infty B^1_sg'(s)ds.
\end{align*}
which exists for $g$ sufficiently smooth and of a certain decay at $\pm\infty$ (note that 
by the law of the
iterated logarithm, $\lv B_t\rv$ is almost surely bounded by $\sqrt{2t\log\log t}$).

Using that Gaussianity of $f$ is equivalent to Gaussianity of all linear combinations of the random variables $\{f(t):t\in\R\}$, it 
is not hard to check that $f(t)_{t\in\R}$ defined in \eqref{process_representation} forms a Gaussian
process on $\R$. Furthermore, it has mean $0$ and covariance function
$-h$. 

By the assumption on the  exponential decay of $\hat{h}$, the rhs of
\eqref{process_representation} can be extended to a strip $\{x+iy:x\in\R, \lv y\rv<c\}$ for some $c>0$ which implies 
analyticity of $f$ in that strip a.s.. Let 
\begin{align}
 D:=(-L-\d,L+\d)\times (-c/2,c/2)\label{def_D}
\end{align}
 with $\d>0$. Then it also follows
from \eqref{process_representation} that the extended process $(f(w))_{w\in D}$ is a complex-valued centered Gaussian process
with covariance function $\E (f(w_1)\overline{f(w_2)})=-h(w_1-\overline{w_2})$.

Recall the abbreviation
\begin{align}
&S_k(t)=\det\left[\frac{\sin(\pi(t_i-t_j))}{\pi(t_i-t_j)}\right]_{1\leq i,j\leq k}\ \text{ and set }\label{S}\\
&\hat{t}_j:=(F_Q^h)^{-1}(t_j/N),\ C(t):=N^{-k}\prod_{j=1}^k\mu(\hat{t}_j)^{-1}.\nonumber
\end{align}
To prove Theorem \ref{universality} (3), in view of \eqref{truncation} we have to show
\begin{align}
&\frac{N!C(t)}{(N-k)!}\E_{N-k,V;L}^N\exp\lee{\U(\hat{t},x)+O_L}\ree-\E_{N,V;L}\exp\lee{\U(x)}\ree S_k(t)=\O\lb
N^{-1+\e}\rb\label{e9}
\end{align}
for any $\e>0$ in the prescribed uniformity. Here we used that by \eqref{concentration_U}, $\E_{N,V}\exp\lee{\U(x)}\ree$ 
is bounded and
bounded away from $0$, which carries over to the truncated setting. Note also the slight abuse of notation by not indicating the 
local 
scaling of the $t_j$'s in $O_L$.

\begin{proof}[Proof of Theorem \ref{universality} (3), negative-definite $h$]
 Let $h$ be negative-definite. Furthermore, let $\tilde{f}$ denote a centered Gaussian process with covariance function $-h$ and 
define $f:=\tilde{f}-\int\tilde{f}d\mu$. Then  we have
\begin{align*}
 &C(t)\frac{N!}{(N-k)!}\E_{{N}-k,V;L}^{N}\exp\lee{\U(\hat{t},x)+O_L}\ree\\
=&C(t)\frac{N!}{(N-k)!}\E
\Big[\E_{N-k,V;L}^{N}
\exp\lee{\sum_{j=1}^{N}f((\hat{t},x)_j)
+O_{L}}\ree\Big]\nonumber\\
=&\E\Big[\E_{N,V;L}
\exp\lee\sum_{j=1}^{N}f(x_j)\ree C(t)R_{N,V,f;L}^k(\hat{t}_1,\dots,\hat{t}_k)\Big],
\end{align*}
where we used the identity
\begin{align*}
 R_{N,V,f;L}^k(\hat{t}_1,\dots,\hat{t}_k)=\frac{N!/(N-k)!}{\E_{N,V,f;L}e^{\sum_{j=1}^N f(x_j)}}\E
\Big[\E_{N-k,V;L}^{N}
\exp\lee{\sum_{j=1}^{N}f((\hat{t},x)_j)
+O_{L}}\ree\Big],
\end{align*}
which can be obtained analogously to \eqref{e8}. Thus the first summand of \eqref{e9} equals
\begin{align}
 \E\Big[\E_{N,V;L}
\exp\lee\sum_{j=1}^{N}f(x_j)\ree\left( C(t)R_{N,V,f;L}^k(\hat{t}_1,\dots,\hat{t}_k)-S_k(t)\right)\Big].\label{toshow}
\end{align}
To apply Theorem \ref{universality} (2), we will replace the integration over all $f$ by an integration over $f$ with $\|f\|_D\leq 
N^\eta$, where $D$ is the complex domain defined in \eqref{def_D}. More precisely, we will show that 
\begin{align}
\E\,&\dopp{1}_{\{\|f\|_D>N^\eta\}}\Big[\E_{N,V;L}
\exp\lee\sum_{j=1}^{N}f({x_j})\ree\nonumber\\
&\times\Big(C(t)R_{N,V,f;L}^k(\hat{t}_1,\dots,\hat{t}_k)- 
S_k(t)\Big)\Big]=\O(e^{-cN^{2\eta}})\label{Gaussian_truncation}
\end{align}
for some $c>0$. This will be done by applying H\"older's inequality to separate
the
expectations of $\dopp{1}_{\{\|f\|_D> N^\eta\}}$, of $\E_{N,V;L}
\exp\lee\sum_{j=1}^{N}f({x_j})\ree$ and of 
\begin{align*}
C(t)R_{N,V,f;L}^k(\hat{t}_1,\dots,\hat{t}_k)- S_k(t).
\end{align*}
Let us reconsider the complex-valued process $(\tilde{f}(t))_{t\in D}$, which is the extension of the Gaussian process $\tilde{f}$ 
on $[-L,L]$ with covariance function $-h$. It follows from \eqref{process_representation} that real and imaginary parts of 
$\tilde{f}$ on $D$ are (real-valued) centered Gaussian
processes. Their covariance functions are readily computed as 
\begin{align*}
 &\E(\Re \tilde{f}(w_1)\Re \tilde{f}(w_2))=-\frac 12(\Re\,h(w_1-w_2)+\Re\, h(w_1-\overline{w_2})),\\
&\E(\Im \tilde{f}(w_1)\Im \tilde{f}(w_2))=\frac 12(\Re\, h(w_1-w_2)-\Re\, h(w_1-\overline{w_2})),\quad w_1,w_2\in D,
\end{align*}
giving the variances
\begin{align*}
 &\E(\Re \tilde{f}(w))^2= -\frac{1}{2}(\Re\,h(0)+\Re\, h(2i\Im w)),\\
&\E(\Im \tilde{f}(w))^2= \frac{1}{2}(\Re\,h(0)-\Re\, h(2i\Im w)), \quad w\in D.
\end{align*}
Now Borell's inequality (see e.g. \cite[Theorem 2.1.1]{AdlerTaylor}) states that the supremum $\|X\|_K$
of a continuous centered Gaussian process $X_t$ over a compact $K$ has sub-Gaussian tails, more precisely it is dominated by a 
Gaussian 
random variable
with a certain expectation and variance $\s:=\sup_{t\in K}\E
X_t^2$. Since the sum of sub-Gaussian variables is sub-Gaussian as well, we see with $\sup_{w\in D}\lv f^+(w)\rv\leq
\sup_{w\in \bar{D}}\lv \Re f^+(w)\rv+\sup_{w\in \bar{D}}\lv \Im f^+(w)\rv$ that  $\sup_{w\in D}\lv f^+(w)\rv$ is also sub-Gaussian. 
The
same reasoning leads to the conclusion that $\sup_{w\in D}\lv f(w)\rv$ is sub-Gaussian, giving
\begin{align}\label{Gaussian_LD}
 P\{\|f\|_D>N^\eta\}=\O(e^{-cN^{2\eta}})
\end{align}
for some $c>0$.

Next, we provide a bound for $\E_{N,V;L}
\exp\lee\sum_{j=1}^{N}f({x_j})\ree$.
Proposition \ref{Concentration} also holds for the ensemble truncated to $[-L,L]$ with an exponentially small error which we will
neglect. However, it is crucial that in the truncated case the
Lipschitz constant is taken over $[-L,L]$ instead of the whole real line.
Thus we get
\begin{align}
 \E_{N,V;L}
\exp\lee\sum_{j=1}^{N}f({x_j})\ree\leq \exp\lee{\frac{(\Lip{f})^2}{2\a_V}}+
 C(\|f\|_{[-L,L]}+\|f^{(3)}\|_{[-L,L]})\ree\label{e10}
\end{align}
for some $C$.
$\tilde{f}'$ is again a centered stationary Gaussian process with covariance function $-h''$ and thus, by similar arguments as 
above, sub-Gaussianity of $\Lip{f}$ and analogously also
sub-Gaussianity of $\|f^{(3)}\|_{[-L,L]}$ follow. Hence for some $\l>1$ (close to 1, coming from H\"older's inequality)
there is a constant $C$ such that for all $N$
\begin{align}
 \E\Big[\E_{N,V;L}
\exp\lee\sum_{j=1}^{N}f({x_j})\ree\Big]^\l<C,\label{linear_LD}
\end{align}
if $\a_Q$ (and hence $\a_V$) is large enough.

It is important to note that, as the processes are stationary on $\R$, the variances of $\Lip{f},\|f\|_{[-L,L]}$ and 
$\|f^{(3)}\|_{[-L,L]}$
and therefore the required $\a_Q$ are independent of the truncation threshold
$L$ and $k$.

Now we will estimate $C(t)R_{N,V,f;L}^k(\hat{t}_1,\dots,\hat{t}_k)- S_k(t)$. If $t_i=t_j$ for some $i\not=j$, then this quantity is 
0, hence we will only consider $t$ with distinct elements. For such $t$, \eqref{determinantal_relations} and 
$P_{N,V,f;L}(x)>0$ for any $x\in [-L,L]^N$ with 
distinct components, imply that 
$(K_{N,V,f;L}(t_i,t_j))_{1\leq i,j\leq k}=:A$ is a positive
definite matrix
and can hence be
written as $A=B^2$ for some matrix $B$. Using Hadamard's inequality then gives
\begin{align*}
 &\det A=\lb \det B\rb^2\leq \prod_{j=1}^k\sum_{i=1}^k\lv B_{ij}\rv ^2=\prod_{j=1}^kA_{jj},
\end{align*}
which is
\begin{align}
 &R_{N,V,f;L}^k\lbb \hat{t}_1,\dots,\hat{t}_k\rbb\leq \prod_{j=1}^kK_{N,V,f;L}(\hat{t}_j,\hat{t}_j)=
\prod_{j=1}^k
R_{N,V,f;L}^{1}(\hat{t}_j).\label{Hadamard}
\end{align}
 Let us employ the representation as inverse Christoffel function (see e.g. 
\cite{Totik})
\begin{align*}
&R^1_{N,V,f}(t)= \frac{e^{-NV+f}}{\l_N(e^{-NV+f},t)},\\
&\l_N(e^{-NV+f},t):=\inf_{P_{N-1}(t)=1}\int \lv
P_{N-1}(s)\rv^2e^{-NV(s)+f(s)}ds,
\end{align*}
where the infimum is taken over all polynomials $P_{N-1}$ of at most degree $N-1$ with $P_{N-1}(t)=1$. This representation 
immediately implies the bound
\begin{align*}
 R_{N,V,f;L}^{1}(\hat{t}_j)\leq R_{N,V;L}^{1}(\hat{t}_j)e^{2\|f\|_{[-L,L]}}.
\end{align*}
To bound $R_{N,V;L}^{1}$, we can now use the uniform convergence stated in Theorem \ref{universality} (2) together with the 
boundedness of the sine kernel, giving
\begin{align}
 &N^{-1}R_{N,V,f;L}^{1}(\hat{t}_j)\leq C e^{2\|f\|_{[-L,L]}}\ \text{ and hence}\nonumber\\
&C(t)R_{N,V,f;L}^k\lbb
\hat{t}_1,\dots,\hat{t}_k\rbb\leq C'e^{2k\|f\|_{[-L,L]}},\label{corrinequality}
\end{align}
where $C,C'$ do not depend on $f$.
As $\|f\|_{[-L,L]}$ is sub-Gaussian, we have for some $\l'>1$ and some $C''$ that for all $N$
\begin{align}
 \E\lv C(t)R_{N,V,f;L}^k(\hat{t}_1,\dots,\hat{t}_k)- S_k(t) \rv^{\l'}\leq C''.\label{correlation_LD}
\end{align}
Combining \eqref{Gaussian_LD}, \eqref{linear_LD} and \eqref{correlation_LD}, it follows that \eqref{Gaussian_truncation} is of
order $\O(e^{-cN^{2\eta}})$. Thus Theorem \ref{universality} (2) yields
\begin{align*}
 \E\Big[\E_{N,V;L}
\exp\lee\sum_{j=1}^{N}f({x_j})\ree\Big( C(t)R_{N,V,f;L}^k(\hat{t}_1,\dots,\hat{t}_k)- S_k(t)\Big)\Big]=\O(N^{-1+\e}),
\end{align*}
for any $\e>0$, given that $\a_Q$ is large enough.
\end{proof}

\subsection{Proof of Theorem \ref{universality} (3) for general $h$}\noindent

A general $h$ may be decomposed into positive-definite functions as follows. Let $_\pm$ denote positive and non-positive part and 
write $\hat{h}=(\hat{h})_+-(\hat{h})_-$. Then $h=h^+-h^-$, where $h^\pm:=\widehat{(\hat{h})_\pm}$ and furthermore, $h^\pm$ are 
positive-definite. To use $h^\pm$ for our linearization method, we need these functions to be real-analytic. The real-analyticity 
is 
somewhat surprisingly equivalent to the exponential decay of $\hat{h}$ at infinity. On the one hand, exponential decay of $\hat{h}$ 
allows to extend  $h^\pm$ to the complex plane via Fourier inversion, thereby showing real-analyticity. On 
the 
other hand, \cite[Theorem 2]{LukacsSzasz} tells us that any real-analytic positive-definite function has a Fourier transform of 
exponential decay.

Instead of $-h$, we will use $-h_z:=zh^++h^-$, $z>0$, as a covariance function and use complex analysis to show the desired.
  To this end define for complex $z\in\C$
\begin{align}
 \U_z(x):=&\frac{z}{2}\lbb\sum_{i,j=1}^Nh^+(x_i-x_j)-\left[h^+_\mu(x_i)+h^+_\mu(x_j)-h^+_{\mu\mu}\right]\rbb\label{e47}\\
+&\frac{1}{2}\lbb\sum_{i,j=1}^Nh^-(x_i-x_j)-\left[h^-_\mu(x_i)+h^-_\mu(x_j)-h^-_{\mu\mu}\right]\rbb\label{e48}.
\end{align}
Again, we have to show
\begin{align}
&\frac{N!C(t)}{(N-k)!}\E_{N-k,V;L}^N\exp\lee{\U_z(\hat{t},x)+O_L}\ree-\E_{N,V;L}\exp\lee{\U_z(x)}\ree S_k(t)=o(1)\label{U_z}
\end{align}
for $z=-1$. The proof for negative-definite $h$ implies \eqref{U_z} for $z>0$. This is basically enough, as Vitali's theorem 
implies, which we state for the convenience of the reader (cf. \cite[5.21]{Titchmarsh}): Let $g_n(z)$ be a sequence of analytic 
functions on a domain $U\subset\C$ with $\lv g_n(z)\rv\leq M$ for all $n$ and all $z\in U$. Assume that $\lim_{n\to\infty}g_n(z)$ 
exists for a set of $z$ having a limit point in $U$. Then $\lim_{n\to\infty}g_n(z)$ exists
for all $z$ in the interior of $U$ and the limit is an analytic function in $z$.

The set containing a limit point will be chosen as a small interval $(0,\d)$ for some $\d>0$. We remark in passing that $\d$ can be 
arbitrarily small and as a consequence $h^+$ has no influence on the necessary size of the convexity constant $\a_Q$.

To transfer the required uniformity in the $t_j$ from the case $z>0$ to $z=-1$ is a technical issue as taking absolute values and 
suprema would destroy the analyticity in $z$, which is necessary for the application of Vitali's theorem. Therefore, we will use the 
following characterization of uniform convergence in terms of sequences:
A sequence of continuous real-valued functions $(g_n)_n$, defined on $\R^l$, converges uniformly on the
sequence of compact sets $(B_n)_n$, $B_n\subset \R^l$ towards a continuous function $g$ if and only if for all sequences
$(n_m)_m\subset \mathbb{N}$ with $\lim_{m\to\infty}n_m=\infty$ and all sequences $(t_m)_m$ with $t_m\in B_{n_m}$ we have
$\lim_{m\to\infty}g_{n_m}(t_m)-g(t_m)=0$.

We will take $B_n:=I_N$. Let $(N_m)_m\subset\N$ be a
sequence going to
infinity and $(t_{(m)})_m$ be a sequence with $t_{(m)}\in I_{N_m}$.
Let us define $\map{W_m}{\C}{\C}$ as
\begin{align}
 W_m(z):=&C(t_{(m)})\E_{{N_m}-k,V;L}^{N_m}\exp\lee{\U_{z}(\hat{t}_{(m)},x)+O_L}\ree\nonumber\\
&-\E_{N_m,V;L}\exp\lee{\U_{z}(x)}
\ree S_k(t_{(m)}).\label{W_m}
\end{align}

Note that we suppressed some of the $m$-dependencies. 
Clearly, $(W_m)_m$ is a sequence of
analytic functions.

\begin{remark}\label{remark_Vitali}
 Vitali's theorem allows to deduce convergence in a region of the complex plane from convergence in another region. As rates of 
convergence might well be different according to the region one is looking at, it is clear that we cannot transfer the rate 
$\O(N^{-1+\e})$, valid for $z>0$, to $z=-1$ with the same technique. This seems to be  rather a technical issue, we in fact 
believe 
that the correct rate should also be (at least) $\O(N^{-1+\e})$.
\end{remark}

\begin{proof}[Proof of Theorem \ref{universality} (3), general $h$]
We will often drop the dependence on $m$ in the following. For positive $z$, we can apply the linearization procedure as described 
above, as then $h_z$ is a positive-definite function. Thus we have $W_m(z)=o(1)$ for any $z\in(0,\d)$ for some $\d$ provided $\a_Q$ 
is large enough, where $\d$ is chosen so small that the lower bound on $\a_Q$ does not depend on $\d$. Hence, to apply Vitali's 
theorem, we need to show boundedness uniform in $m$ and $z$ from a domain containing $-1$ and $(0,\d)$. This domain may 
in fact be chosen as the halfplane $\{z\in\C\,:\,\Re z<\d\}$, which can be seen as follows. Bounding $W_m$ termwise, 
\eqref{e47} shows that $\Im z$ only gives a phase which vanishes by taking absolute values, hence we can concentrate on real $z$. 
For $z<0$, \eqref{e47} is non-positive, since it is minus $1/2$ times the variance of a Gaussian random variable (cf. \eqref{e7}), 
implying that in this case the influence of \eqref{e47} in bounding $W_m$ will vanish as well, due to taking absolute values. It is 
thus sufficient to consider $z>0$. In this case we get as above that $W_m(z)$ equals
\begin{align}
\E\Big[\E_{N,V;L}
\exp\lee\sum_{j=1}^{N}f_z({x_j})\ree\Big(C(t)R_{N,V,f_z;L}^k(\hat{t}_1,\dots,\hat{t}_k)- S_k(t)\Big)\Big],\label{e11}
\end{align}
where $f_z$ is a centered Gaussian process with covariance function $-h_z$. Now, the bounds \eqref{e9} and \eqref{corrinequality} 
can be used again to show uniform boundedness in $m$. To check that these bounds are uniform in $z\in(0,\d)$ for $\d>0$ small, is 
straightforward and left to the reader. The theorem is proved.
\end{proof}

\section{Proofs of remaining statements}\label{sec_remaining_res}

\begin{proof}[Proof of Theorem \ref{thrmspacings}]
Let an interval $I_N\subset [0,N]$ be given such that\\  $\liminf_{N\to\infty}\textup{dist}(I_N,\{0,N\})/N>0$. We will apply 
Theorem 
\ref{thrmgeneral} with $P_N$ being the distribution of the \textit{unfolded} ensembles $P_{N,V,f}$ and 
$P_{N,Q}^h$, respectively. We start with the former case. Let $F_V^{[-1]}$ be a function with the following properties:
\begin{enumerate}
 \item $\map{F_V^{[-1]}}{\R}{J}$ strictly monotonically increasing, continuously differentiable,
\item $F_V^{[-1]}(\R)=J$,
\item $F_V^{[-1]}(t)=F_V^{-1}(t)$ for $t\in U$ with $U\supset \liminf_{N\to\infty}I_N/N$ open.
\end{enumerate}
Now, we define
\begin{align*}
P_N(x):=\prod_{j=1}^N {F_V^{[-1]}}'(x_j)P_{N,V,f}(F_V^{[-1]}(x_1),\dots,F_V^{[-1]}(x_N)).
\end{align*}
By properties (1) and (2), $P_N$ is a probability measure on $\R^N$. By property (3), we have 
${F_V^{[-1]}}'(t)=(\mu_V(F_V^{-1}(t)))^{-1}$ for all $t$ such that $Nt\in I_N$ for $N$ large enough.
Then the correlation function $R_{N}^k$ is (for $N$ large enough) on $I_N^k$ given by
\begin{align*}
 R_N^k(t_1/N,\dots,t_k/N)=\prod_{j=1}^k \mu_V(F_V^{-1}(t_j/N))^{-1}R_{N,V,f}^k(F_V^{-1}(t_1/N),\dots, F_V^{-1}(t_k/N)).
\end{align*}

It remains to check conditions (1) and (2) in Theorem \ref{thrmgeneral}. For condition (1), we 
may use inequality \eqref{corrinequality}, giving
\begin{align*}
\prod_{j=1}^k \mu_V(\hat{t}_j)^{-1}R_{N,V,f}^k\lbb\hat{t}_1,\dots,\hat{t}_k\rbb\leq C'\prod_{j=1}^k 
\mu_V(\hat{t}_j)^{-1}e^{2k\|f\|_\infty},
\end{align*}
which is valid also for ensembles with non-compact support. Here we use again the abbreviation $\hat{t}_j:=F_V^{-1}(t_j/N)$.
Positivity of $\mu_V$ on $(a_V,b_V)$ yields condition (1) for the 
unitary invariant ensembles. 

For condition (2), another application of Hadamard's inequality may be used, which we cite from \cite[Lemma 4.3.2]{AGZ}. For 
kernels $K_1, K_2$, defined on some locally compact $A\times A$ and all $t\in A^k$, the following inequality holds:
\begin{align}
 \left\lv \det (K_1(t_i,t_j))_{1\leq i,j\leq k}-\det (K_2(t_i,t_j))_{1\leq i,j\leq k}\right\rv\leq 
k^{k/2+1}\|K_1-K_2\|\cdot\max(\|K_1\|,\|K_2\|)^{k-1}.\label{Hadamard2}
\end{align}
Here $\|K\|:=\sup_{t,s\in A}\lv K(t,s)\rv$ denotes the sup-norm of a kernel $K$ on $A\times A$. Choosing 
$K_1(t,s):=(N\mu_V(F_V^{-1}(t/N)))^{-1}K_{N,V,f}(F_V^{-1}(t/N),F_V^{-1}(s/N))$ and $K_2$ as the sine kernel, we find with $A:=I_N$ 
that by Theorem \ref{thrmgeneral} (1)
$\|K_1-K_2\|=\O(1/N)$. Moreover, the boundedness of the sine kernel and the uniform convergence of $K_1$ imply that 
$\max(\|K_1\|,\|K_2\|)\leq C$ for some $C>1$. This proves condition (2) of Theorem \ref{thrmgeneral} for $P_{N,V,f}$.

For the repulsive particle systems $P_{N,Q}^h$, we will first truncate the ensemble to a compact $[-L,L]^N$ and apply Theorem 
\ref{thrmgeneral} to the truncated ensemble. This will be convenient as the truncation threshold $L$ in \eqref{truncation} depends 
on the order of the correlation function $k$ and in the subsequent linearization several constants depend on $L$, which would 
complicate checking conditions (1) and (2) of Theorem \ref{thrmgeneral} significantly.

For any $L>0$ we have the crude bound
\begin{align*}
  &\E_{N,Q}^h\left( \dopp{1}_{([-L,L]^N)^c}(x)\sup_{s \in \mathbb R} \left| \frac1{\lv I_N\rv}\int_0^s d\s(I_N,\tilde{x}) - G(s) 
\right|\right)\leq N P_{N,Q}^h\left(([-L,L]^N)^c\right)\\
&\leq N\int_L^\infty R_{N,Q}^{h,1}(t)dt.
\end{align*}
Now, \eqref{truncation_Q} tells us that choosing $L$ large enough, the last expression is bounded by $\exp(-N)$, hence negligible. 
This also shows that replacing the normalizing constant $Z_{N,Q}^h$ by its truncated variant $Z_{N,Q;L}^h$, obtained by replacing 
the integrations in $Z_{N,Q}^h$ by integrations over $[-L,L]$, only results in an error of at most $\exp(-N)$. Summarizing, we get
\begin{align*}
 &\E_{N,Q}^h\left(\sup_{s \in \mathbb R} \left| \frac1{\lv I_N\rv}\int_0^s d\s(I_N,\tilde{x}) - G(s) 
\right|\right)\\
&=\E_{N,Q;L}^h\left(\sup_{s \in \mathbb R} \left| \frac1{\lv I_N\rv}\int_0^s 
d\s(I_N,\tilde{x}) - G(s) 
\right|\right)+\O(e^{-N}),
\end{align*}
where $\E_{N,Q;L}^h$ denotes expectation w.r.t.~ the truncated ensemble $P_{N,Q;L}^h$. Note that both ensembles have the 
same limiting measure $\mu_Q^h$ and therefore, there is no change in the unfolding of the particles. The case of $\hat{\s}$ instead 
of $\lv I_N\rv^{-1}\s$ is completely analogous.

The same procedure as above can now be applied to fit $P_{N,Q;L}^h$ into the framework of Theorem \ref{thrmgeneral}.
Condition (1) then follows using the linearization procedure introduced in 
Section 5. Indeed, an inequality of the form
\begin{align*}
 R_{N,Q}^{h,k}\lbb\hat{t}_1,\dots,\hat{t}_k\rbb\leq C
\end{align*}
for some $C$, uniformly in $t_1,\dots,t_k$, is equivalent to 
\begin{align}
 \E_{N-k,V;L}^N\exp\lee{\U_{z}(\hat{t},x)+O_L}\ree\leq C'\label{inequU}
\end{align}
uniformly in $t_1,\dots,t_k$ for small positive $z$. By linearization, the l.h.s.~of \eqref{inequU} is equal to 
\begin{align*}
 \E\Big[\E_{N,V;L}
\exp\lee\sum_{j=1}^{N}f(x_j)\ree R_{N,V,f;L}^k(\hat{t}_1,\dots,\hat{t}_k)\Big],
\end{align*}
to which now \eqref{corrinequality} may be applied for almost all $f$. The sub-Gaussianity of $\|f\|_\infty$ gives the desired 
bound.

 For condition (2), we can use the same arguments and \eqref{Hadamard2}.
\end{proof}

\begin{proof}[Proof of Corollary \ref{cor_spacings}]
Let $I_N$ be a sequence of proper compact sub-intervals of $[0,N]$ with $I_N/N\to[0,1]$. We will use the simple inequality
\begin{align}
 &\E_N\sup_{s\in\R}\left\lv\frac{1}{N-1}\int_0^sd\s([0,N],\tilde{x})-G(s)\right\rv\nonumber\\
&\leq  
\E_N\sup_{s\in\R}\left\lv\frac{1}{N-1}\int_0^sd\s(I_N,\tilde{x})-G(s)\right\rv+\E_N\frac{1}{N-1}\int_0^\infty d\s([0,N]
\setminus I_N,\tilde{x}),\label{transitionestimate}
\end{align}
 The last term is 
\begin{align*}
 \frac{1}{N-1}\lb\E_N\#\left\{j\,:\,x_j\in F^{-1}\lb\frac{[0,N]\setminus I_N}{N}\rb\right\}-1\rb,
\end{align*}
where $F=F_V$ or $F=F_Q^h$, respectively. $F^{-1}(([0,N]\setminus I_N)/{N})$ consists of the two intervals $(-\infty,a_V+\e_N')$ 
and $(b_V-\e_N,\infty)$ for two sequences of positive numbers $\e_N, \e_N'$, both converging to 0. Let us exemplarily deal with 
$(b_V-\e_N,\infty)$. We have
\begin{align}
 \frac{1}{N-1}\E_N\#\left\{j\,:\,x_j\in (b_V-\e_N,\infty) \right\}=\frac{1}{N-1}\int_{b_V-\e_N}^\infty 
R^1_N(t)dt.\label{estimate1}
\end{align}
For $R^1_{N,V,f}$, \cite[Theorem 1.5]{KSSV} provides the following asymptotics. For a certain $c>0$ and all $t\in \hat{J}$ with
$t>1+c^{-1}N^{-2/3}$, we have
\begin{align}
 &R^1_{N,V,f}(\l_{V,f}(t))\nonumber\\
&=\frac{1}{2\pi(b_{V,f}-a_{V,f}) N}e^{-N\eta_{V,f}(t)}\lb\frac{1}{t^2-1}+\O\lb\frac{1}{N( 
t-1)^{5/2}}\rb+\O\lb\frac{1}{N}\rb\rb,\label{boundcorrelation}
\end{align}
where 
\begin{align}
 \eta_V(t):=\int_1^t\sqrt{s^2-1}G_V(s)ds\label{def_eta}
\end{align}
 and $G_V$ has been defined in \eqref{def_G}. Similarly 
to \eqref{approx1}, \eqref{approx2} and \eqref{approxrho}, we can show $\eta_{V,f}=\eta_V(1+\O(\|f\|/N))$. As $V'$ is strictly 
increasing and $\lim_{t\to\infty}V(t)=\infty$, we have for all $t$ large enough and $s\in[-1,1]$ that 
$(V\circ\l_V)'(t)-(V\circ\l_V)'(s)\geq c'$ for some $c'>0$. Thus (cf. \eqref{def_h}) $h_V(t,s)\geq \dfrac{c'}{t-s}$ and  (cf. 
\eqref{def_G})
\begin{align*}
 G_V(t)\geq \frac{c'}{\pi}\int_{-1}^1\frac{1}{(t-s)\sqrt{1-s^2}}ds=\frac{c'}{\sqrt{t^2-1}},
\end{align*}
where the last equality is due to $t>1$. Hence $\eta_V(t)\geq c't$ and we conclude that for some $c''>0$
\begin{align}
 \int_{b_V+c''N^{-2/3}}^\infty R_{N,V,f}^1(t)dt=o(1).\label{estimate2}
\end{align}
With the linearization technique, this bound can be transfered to the ensemble $P_{N,Q}^h$. We note in passing that due to the 
mixing over $f$, for $P_{N,Q}^h$ the bound for the first correlation function is somewhat worse than \eqref{boundcorrelation}, see 
\cite{KV} for details.

For the edge regime, i.e. $b_V-\e_N<t<b_V+c''N^{-2/3}$, the following asymptotics are given in \cite[Theorem 1.5]{KSSV}, valid for 
$1-c<t<1+cN^{-2/5}$ (with the same $c$ as above),
\begin{align*}
 &R^1_{N,V,f}(\l_{V,f}(t))=\frac{2N^{-1/3}\g_{V,f}}{b_{V,f}-a_{V,f}}\KAi\lb 
\g_{V,f}N^{2/3}(t-1),\g_{V,f}N^{2/3}(t-1)\rb(1+r(t)),
\end{align*}
where $\g_{V,f}:=2^{-1/3}G_{V,f}(1)^{2/3}$,
\begin{align*}
 r(t)=\begin{cases}
	\O(1-t)+\O(N^{-2/3}), &\text{ if } t\leq 1\\
	\O(N(t-1)^{5/2})+\O(N^{-2/3}),  &\text{ if } t\geq 1
      \end{cases}
\end{align*}
is uniform in $f$ as above and as before, we can neglect the $f$-dependence. Here $\KAi$ is the Airy kernel,
\begin{align*}
 \KAi(t,s):=\frac{\Ai(t)\Ai'(s)-\Ai'(t)\Ai(s)}{t-s}.
\end{align*}
On the diagonal, the 
Airy kernel is given as $\KAi(t,t)=\Ai'(t)^2-t\Ai(t)^2$. 
We are therefore left 
to estimate
\begin{align*}
 R_{N,V}^1(t)=\frac{2N^{-1/3}\g_{V}}{b_V-a_V}\KAi\lb \g_{V}N^{2/3}(\l_V^{-1}(t)-1),\g_{V}N^{2/3}(\l_V^{-1}(t)-1)\rb.
\end{align*}
  Let us first consider the case $t<b_V$, i.e. $\l_V^{-1}(t)<1$. Setting 
$\zeta(t):=\frac{2}{3}t^{3/2}$, \cite[10.4.60, 10.4.62]{AbramowitzStegun} provide the following asymptotics, valid as $t>0$,
\begin{align}
 &\Ai(-t)=\pi^{-1/2}t^{-1/4}\lb \sin\lbb \zeta(t)+\frac\pi4\rbb+\O(\zeta(t)^{-1})\rb,\label{Airyasymptotics}\\
&\Ai'(-t)=-\pi^{-1/2}t^{1/4}\lb \cos\lbb \zeta(t)+\frac\pi4\rbb+\O(\zeta(t)^{-1})\rb,
\label{Airy'asymptotics}
\end{align}
from which we conclude
\begin{align*}
 \Ai(-t)^2=\O(t^{-1/2}), \ \Ai'(-t)^2=\O(t^{1/2}).
\end{align*}
Thus we get uniformly for $b_V-\e_N<t\leq1$
\begin{align*}
 N^{-1/3}\KAi\lb \g_{V}N^{2/3}(\l_V^{-1}(t)-1),\g_{V}N^{2/3}(\l_V^{-1}(t)-1)\rb=\O(\l_V^{-1}(t)-1)=o(1).
\end{align*}
For $1\leq t\leq 1+cN^{-2/5}$, the formulae \cite[10.4.59, 10.4.61]{AbramowitzStegun} provide
\begin{align*}
  \Ai(t)^2=\O(t^{-1/2}e^{-\frac43 t^{3/2}}), \ \Ai'(t)^2=\O(t^{1/2}e^{-\frac43 t^{3/2}}), 
\end{align*}
which yields uniformly for $b_V\leq t<b_V+c''N^{-2/5}$
\begin{align*}
 R_{N,V}^1(t)=o(1)
\end{align*}
in an analogous way. This gives 
\begin{align}
 \int_{b_{V}-\e_N}^{b_V+c''N^{-2/5}}R_{N,V,f}^1(t)dt=o(1).\label{estimate3}
\end{align}
This estimate can by the now familiar procedure be extended to the ensemble $P_{N,Q}^h$. Altogether we have combining 
\eqref{estimate1}, \eqref{estimate2} and \eqref{estimate3}
\begin{align*}
 \frac{1}{N-1}\E_N\#\left\{j\,:\,x_j\in (b_V-\e_N,\infty) \right\}=o(1).
\end{align*}
Returning to \eqref{transitionestimate}, we wish to apply Theorem \ref{thrmgeneral} for an interval $I_N$ which exhausts $[0,N]$ for $N\to\infty$. It suffices to show uniform convergence of
the correlation functions towards the sine kernel, as conditions (1) and (2) of Theorem \ref{thrmgeneral} will then follow exactly 
as in \eqref{corrinequality} and \eqref{Hadamard2}. Hence our task is to extend the proof of Theorem \ref{universality} to regions 
close 
to the edges.

 Formula \eqref{bulk-formula} was valid for $r,s\in(-1+\d,1-\d)$ for some arbitrary but fixed 
$\d>0$. As now $I_N$ covers some of the left and right edge regions, we also have to consider correlations 
between particles from 
different regions.
  Here the general \cite[Theorem 1.3]{KSSV} is useful which gives 
\begin{align*}
 K_{N,V,f}(t,s)=\frac{k_1(t)k_2(s)-k_2(t)k_1(s)}{t-s}+\O(N^{-1}),
\end{align*}
where $k_1,k_2$ are bounded functions and the $\O$ term is uniform for $t,s$ from bounded subsets of $J$. This means that $K_{N,V,f}(t,s)$ is bounded if $t-s$ is bounded away from 0, which
includes the case of one particle at the left and the other at the right edge. Dividing by $N$, we see that Theorem \ref{universality} (1) and (2) hold in this case. The correlations between bulk particles and 
edge particles are more subtle, as these regions are adjacent and a transition from sine kernel to Airy kernel statistics occurs.
Without loss of generality we will consider this transition at the right edge only. 
To state the analogue of \eqref{bulk-formula} at the edge, we need some more notation. Let us remark that we restrict ourselves to 
the case $t\leq b_V$, the void region not being of interest here thanks to \eqref{estimate2} and \eqref{estimate3}.

Recalling \eqref{def_rho}, define
\begin{align*}
 \map{\xi_V}{\R}{\R}, \ \xi_V(t):=2\pi\int_t^1\rho_V(s)ds
\end{align*}
and for some $0<\d<1$
\begin{align}
 &\map{f_{N,V}}{[1-\d,1]}{\R},\ f_{N,V}(t):=-N^{2/3}\lb \frac{3}{4}\xi_V(t)\rb^{2/3},\label{def_f}\\
&\map{d_V}{[1-\d,1]}{\R},\ d_V(t):=\hat{a}(t)^{-1}\g_V^{-1/4}\lb \frac34\xi_V(t)\rb^{1/6}.\label{def_d}
\end{align}

 Abbreviating $f_N:=f_{N,V}$, \cite[Proposition 4.1]{KSSV} states that for some $\d>0$ uniformly for $r,s\in 
[1-\d,1]$
\begin{align*}
 &\frac{b_{V,f}-a_{V,f}}{2}K_{N,V,f}(\l_{V,f}(r),\l_{V,f}(s))=\KAi(f_N(r),f_N(s))\frac{f_N(r)-f_N(s)}{r-s}\\
&+\lb 
\frac{\Ai(f_N(r))\Ai'(f_N(s))}{d_V(s)}+\frac{\Ai(f_N(s))\Ai'(f_N(r))}{d_V(r)}\rb\frac{d_V(r)-d_V(s)}{r-s}+\O(1/N).
\end{align*}
Note that this formula covers a part of the bulk region.
By \eqref{Airyasymptotics} and \eqref{Airy'asymptotics},
\begin{align}
&\KAi(f_N(r),f_N(s))(f_N(r)-f_N(s))\nonumber
\\&=\Ai(f_N(r))\Ai'(f_N(s))-\Ai'(f_N(r))\Ai(f_N(s))\label{Airyproducts}\\
&=\frac1\pi\Big[\cos\lb\zeta(f_N(s))-\frac\pi4\rb\cos\lb\zeta(f_N(r))+\frac
\pi4\rb\nonumber\\
&-\cos\lb\zeta(f_N(r))-\frac\pi4\rb
\cos\lb\zeta(f_N(s))+\frac\pi4\rb\Big]+\O\lb\zeta(\lv f_N(r)\rv)^{-1}+\zeta(\lv f_N(s)\rv)^{-1}\rb\nonumber\\
&=\frac1\pi\sin(\zeta(f_N(r))- \zeta(f_N(s)))+\O\lb\zeta(\lv f_N(r)\rv)^{-1}+\zeta(\lv f_N(s)\rv)^{-1}\rb\nonumber\\
&=\frac1\pi\sin\lb 
N\pi\int_s^r\rho_V(u)du\rb+\O\lb \frac{\xi_V(r)^{-1}+\xi_V(s)^{-1}}{N}\rb.\nonumber
\end{align}
Hence we have for $r,s<1$, $r\not=s$,
\begin{align}
 &\frac{b_{V,f}-a_{V,f}}{2}K_{N,V,f}(\l_{V,f}(r),\l_{V,f}(s))=\frac1\pi\frac{\sin\lb 
N\pi\int_s^r\rho_{V,f}(u)du\rb}{r-s}\nonumber\\
&+\lb 
\frac{\Ai(f_N(r))\Ai'(f_N(s))}{d_V(s)}+\frac{\Ai(f_N(s))\Ai'(f_N(r))}{d_V(r)}\rb\frac{d_V(r)-d_V(s)}{r-s}\label{restAiry}\\
&+\O\lb \frac{\xi_{V,f}(r)^{-1}+\xi_{V,f}(s)^{-1}}{N(r-s)}\rb+\O(1/N).\nonumber
\end{align}
We have thus recovered the leading term of \eqref{bulk-formula}. However, the $\O$-term involving $\xi(r)^{-1}$ will only be small 
for $r$ and $s$ being not too close to each other. We will therefore first consider the case of $\lv 
r-s\rv\geq N^{-2+p}$ for some small $p>0$.
From
\eqref{Airyasymptotics}, 
\eqref{def_f}, \eqref{def_d} and the boundedness of the derivative of $d_V$ it follows that for $1-\d<r,s<1-\e'_N$ with $\e'_N>0$ 
converging to 0 slowly enough, we have
\begin{align*}
  &\frac{b_{V,f}-a_{V,f}}{2N}K_{N,V,f}(\l_{V,f}(r),\l_{V,f}(s))=\frac{\sin\lb 
N\pi\int_s^r\rho_{V,f}(u)du\rb}{\pi N(r-s)}+\O(N^{-\iota})
\end{align*}
for some $\iota>0$, uniformly for $r,s$ with $\lv 
r-s\rv\geq N^{-2+p}$.   Now we can proceed as in the proof of Theorem \ref{universality} (1) and (2).  We get with properly chosen 
$\e_N>0$ converging to 0, that 
uniformly in $t,s\in [(1-\d)N,(1-\e_N)N]$
 \begin{align}
\frac{1}{N \mu_V(F_V^{-1} (\frac{t}{N}))} K_{N,V,f} \left( F_V^{-1}\left(\frac{t}{N}\right),F_V^{-1}\left(\frac{s}{N} \right) 
\right)=
\frac{\sin (\pi (t-s))}{\pi (t-s)} + \mathcal O\left(N^{-\iota}\right).\label{AirytoSine}
\end{align}
Here we already replaced $F_{V,f}$ and $\mu_{V,f}$ by their counterparts $F_V$ and $\mu_V$, which has been justified in the proof of 
Theorem \ref{universality}.

For $\lv r-s\rv< N^{-2+p}$, we may use Taylor's expansion in $s$ at the point $r$ in \eqref{Airyproducts} together with $\Ai''(t)=t\Ai(t)$ to obtain
\begin{align}
 &\KAi(f_N(r),f_N(s))(f_N(r)-f_N(s))\nonumber\\
&=f_N'(r)f_N(r)\Ai(f_N(r))^2(s-r)-f_N'(r)\Ai'(f_N(r))^2(s-r)\label{Airy1}\\
&+\frac12\Ai(f_N(r))\big[f_N'(\nu)^2\Ai(f_N(\nu))+f_N(\nu)f_N''(\nu)\Ai(f_N(\nu))\nonumber\\
&\qquad \qquad \qquad+f_N(\nu)f_N'(\nu)^2\Ai'(f_N(\nu))\big](s-r)(s-\nu)\label{Airy2}\\
&+\frac12\Ai'(f_N(r))\left[f_N(\l)f_N'(\l)\Ai(f_N(\l))+f_N''(\l)\Ai'(f_N(\l))\right](s-r)(s-\l)\label{Airy3}
\end{align}
for certain $\nu,\l$ between $s$ and $r$. Using $\KAi(t,t)=\Ai'(t)^2-t\Ai(t)^2$, we see that \eqref{Airy1} equals
\begin{align*}
 f_N'(r)\KAi(f_N(r),f_N(r))(r-s).
\end{align*}
>From \eqref{Airyasymptotics}, \eqref{Airy'asymptotics}, \eqref{def_f} and  $\lv r-s\rv< N^{-2+p}$, we find 
\begin{align*}
 &\frac{b_{V,f}-a_{V,f}}{2N}K_{N,V,f}(\l_{V,f}(r),\l_{V,f}(s))=\frac{f_N'(r)}{N}\KAi(f_N(r),f_N(r))+\O(N^{-1+p})
\end{align*}
with the $\O$ term being uniform in $r,s$. Using \eqref{Airyasymptotics} and \eqref{Airy'asymptotics} again, we can for 
$t\to-\infty$ derive 
\begin{align*}
 \KAi(t,t)=\frac{1}{\pi}(-t)^{1/2}(1+\O((-t)^{-3/2})).
\end{align*}
With these asymptotics and \eqref{def_f} we see that
\begin{align*}
\frac{f_N'(r)}N\KAi(f_N(r),f_N(r))=\rho_V(r)+\O(N^{-\iota})
\end{align*}
uniformly for $1-\d\leq r<1-\e'_N$ with $\e'_N\to\infty$ slowly enough, which establishes \eqref{AirytoSine} also close to the 
diagonal.
As written above, this suffices to apply Theorem \ref{thrmgeneral} to finish the proof. The transfer to correlation functions is 
made with \eqref{determinantal_relations}  and to $P_{N,Q}^h$ with the linearization 
method.
This proves Corollary \ref{cor_spacings}.
\end{proof}

\begin{proof}[Proof of Corollary \ref{Localized Scaling}]
 The corollary follows from Theorem \ref{thrmgeneral} by setting 
\begin{align*}
 P_N(x):=\mu_V(a)^{-N}P_{N,V,f}\lb a+\frac{x_1}{\mu_V(a)},\dots,a+\frac{x_N}{\mu_V(a)}\rb.
\end{align*}
The case of repulsive particles is analogous.
By the two-fold application of Hadamard's inequality (cf. \eqref{Hadamard} and \eqref{Hadamard2}) and the linearization method, it 
suffices to show
\begin{align*}
 \frac{1}{N\mu_V(a)}K_{N,V,f}\lb a+\frac{t}{N\mu_V(a)},a+\frac{s}{N\mu_V(a)}\rb=\frac{\sin(\pi(t-s))}{\pi(t-s)}+\O\lb \frac{1+\lv 
t\rv+\lv s\rv}{N}\rb.
\end{align*}
This is precisely the statement of \cite[Theorem 1.8]{KSSV}, which can also easily be deduced from \eqref{bulk-formula}. Here we 
used again that the $f$-dependence in the scaling can be neglected.
\end{proof}

\begin{proof}[Proof of Corollary \ref{cor_intensity}]
 A careful reading of the proof of Lemma \ref{pointwise} shows that we have for any $s,L$ and any $\e>0$
\begin{align*}
 &\left\lv\int_0^s\frac{1}{\lv I_N\rv}d\E_N\s(I_N,\tilde{x}) - G(s)\right\rv\\
&\leq \sum_{k=2}^L\left[ \frac{s^kC_0^k}{(k-1)!}  \mathcal{O}\left(\frac{1}{|I_N|}\right)  +  
\frac{s^{k-1}   C_0^kk^{k/2+1}}{(k-1)!N^{1-\e}}\right]+E(s,L)+E(s,L+1),
\end{align*}
where $E(s,k)$ has been defined in \eqref{def_E}. Choosing $L\in\mathbb{N}$ such 
that $\sqrt{\log \lv I_N\rv}=o(L)$ and $L=o(\log \lv I_N\rv)$, we get similarly as in the proof of Lemma \ref{pointwise}
\begin{align*}
 &\left\lv\int_0^s\frac{1}{\lv I_N\rv}d\E_N\s(I_N,\tilde{x}) - G(s)\right\rv=\O(\lv I_N\rv^{-1+\e'})
\end{align*}
for any $\e'>0$, where the $\O$ term is uniform in $0<s=\O(\sqrt{\log \lv I_N\rv})$. Thus it remains to see that taking the 
supremum over $[0,\O(\sqrt{\log \lv I_N\rv}))$ is sufficient. It follows from the sub-Gaussian tails of $G$ that $1-G(\O(\sqrt{\log 
\lv I_N\rv}))=\O(\lv I_N\rv^{-1})$ (cf. \eqref{deltabound}). From \eqref{firstmoment} we know that the expected total mass of $\lv 
I_N\rv^{-1}\s(I_N,\tilde{x})$ is $1+\O(\lv I_N\rv^{-1+\e})$ for any $\e>0$ and hence the uniform approximation on $[0,\O(\sqrt{\log 
\lv I_N\rv}))$ gives that 
\begin{align*}
 \int_{\O(\sqrt{\log \lv I_N\rv})}^\infty \frac{1}{\lv I_N\rv}d\E_N\s(I_N,\tilde{x})=\O(\lv I_N\rv^{-1+\e'})
\end{align*}
for any $\e'>0$. This proves the corollary.
\end{proof}

\section*{Acknowledgement}
The second author would like to thank S. Jansen for pointing out that the Gaussian linearization \eqref{e6} is known in 
mathematical physics as the Sine-Gordon transformation.

\bibliographystyle{alpha}
\bibliography{bibliography}

\newcommand{\etalchar}[1]{$^{#1}$}
\begin{thebibliography}{BdMPS95}

\bibitem[AGZ10]{AGZ}
G.W. Anderson, A.~Guionnet, and O.~Zeitouni.
\newblock {\em An introduction to random matrices}, volume 118 of {\em
  Cambridge Studies in Advanced Mathematics}.
\newblock Cambridge University Press, Cambridge, 2010.

\bibitem[AS64]{AbramowitzStegun}
Milton Abramowitz and Irene~A. Stegun.
\newblock {\em Handbook of mathematical functions with formulas, graphs, and
  mathematical tables}, volume~55 of {\em National Bureau of Standards Applied
  Mathematics Series}.
\newblock For sale by the Superintendent of Documents, U.S. Government Printing
  Office, Washington, D.C., 1964.

\bibitem[AT07]{AdlerTaylor}
R.J. Adler and J.E. Taylor.
\newblock {\em Random fields and geometry}.
\newblock Springer Monographs in Mathematics. Springer, New York, 2007.

\bibitem[BdMPS95]{BPS}
A.~Boutet~de Monvel, L.~Pastur, and M.~Shcherbina.
\newblock On the statistical mechanics approach in the random matrix theory:
  integrated density of states.
\newblock {\em J. Statist. Phys.}, 79(3-4):585--611, 1995.

\bibitem[BEY14]{BEY1}
Paul Bourgade, L{\'a}szl{\'o} Erd{\H{o}}s, and Horng-Tzer Yau.
\newblock Universality of general {$\beta$}-ensembles.
\newblock {\em Duke Math. J.}, 163(6):1127--1190, 2014.

\bibitem[BFG15]{Guionnet}
F.~Bekerman, A.~Figalli, and A.~Guionnet.
\newblock Transport {M}aps for {$\beta$}-{M}atrix {M}odels and {U}niversality.
\newblock {\em Comm. Math. Phys.}, 338(2):589--619, 2015.

\bibitem[BGK13]{Borotetal}
G.~{Borot}, A.~{Guionnet}, and K.~K. {Kozlowski}.
\newblock {Large-N asymptotic expansion for mean field models with Coulomb gas
  interaction}.
\newblock {\em ArXiv e-prints}, December 2013.

\bibitem[BGS84]{BGS}
O.~Bohigas, M.-J. Giannoni, and C.~Schmit.
\newblock Characterization of chaotic quantum spectra and universality of level
  fluctuation laws.
\newblock {\em Phys. Rev. Lett.}, 52(1):1--4, 1984.

\bibitem[CGZ14]{Chafaietal}
Djalil Chafa{\"{\i}}, Nathael Gozlan, and Pierre-Andr{\'e} Zitt.
\newblock First-order global asymptotics for confined particles with singular
  pair repulsion.
\newblock {\em Ann. Appl. Probab.}, 24(6):2371--2413, 2014.

\bibitem[DKM{\etalchar{+}}99]{Deiftetal99}
P.~Deift, T.~Kriecherbauer, K.~T.-R. McLaughlin, S.~Venakides, and X.~Zhou.
\newblock Uniform asymptotics for polynomials orthogonal with respect to
  varying exponential weights and applications to universality questions in
  random matrix theory.
\newblock {\em Comm. Pure Appl. Math.}, 52(11):1335--1425, 1999.

\bibitem[Dys62]{Dyson}
Freeman~J. Dyson.
\newblock Statistical theory of the energy levels of complex systems. i.
\newblock {\em Journal of Mathematical Physics}, 3(1):140--156, 1962.

\bibitem[EY12]{ErdosYau}
L.~{Erdos} and H.-T. {Yau}.
\newblock {Gap Universality of Generalized Wigner and beta-Ensembles}.
\newblock {\em ArXiv e-prints}, November 2012.

\bibitem[GV14]{GoetzeVenker}
Friedrich G{\"o}tze and Martin Venker.
\newblock Local universality of repulsive particle systems and random matrices.
\newblock {\em Ann. Probab.}, 42(6):2207--2242, 2014.

\bibitem[Joh01]{Johansson01}
Kurt Johansson.
\newblock Universality of the local spacing distribution in certain ensembles
  of {H}ermitian {W}igner matrices.
\newblock {\em Comm. Math. Phys.}, 215(3):683--705, 2001.

\bibitem[KS99]{KatzSarnak}
Nicholas~M. Katz and Peter Sarnak.
\newblock {\em Random matrices, {F}robenius eigenvalues, and monodromy},
  volume~45 of {\em American Mathematical Society Colloquium Publications}.
\newblock American Mathematical Society, Providence, RI, 1999.

\bibitem[KS03]{KrbalekSeba}
Milan Krbalek and Petr Seba.
\newblock Headway statistics of public transport in {M}exican cities.
\newblock {\em J. Phys. A}, 36(1):L7--L11, 2003.

\bibitem[KS13]{KriecherbauerSchubert}
Thomas Kriecherbauer and Kristina Schubert.
\newblock Spacings: an example for universality in random matrix theory.
\newblock In {\em Random matrices and iterated random functions}, volume~53 of
  {\em Springer Proc. Math. Stat.}, pages 45--71. Springer, Heidelberg, 2013.

\bibitem[KSSV14]{KSSV}
T.~{Kriecherbauer}, K.~{Schubert}, K.~{Sch{\"u}ler}, and M.~{Venker}.
\newblock {Global Asymptotics for the Christoffel-Darboux Kernel of Random
  Matrix Theory}.
\newblock {\em ArXiv e-prints}, January 2014.
\newblock To appear in \textit{Markov Process. Related Fields}.

\bibitem[Kui11]{handbookKuijlaars}
A.~B.~J. Kuijlaars.
\newblock Universality.
\newblock In {\em The {O}xford handbook of random matrix theory}, pages
  103--134. Oxford Univ. Press, Oxford, 2011.

\bibitem[KV15]{KV}
T.~Kriecherbauer and M.~Venker.
\newblock Edge statistics for a class of repulsive particle systems.
\newblock {\em ArXiv e-prints}, January 2015.

\bibitem[LL08]{LL08}
E.~Levin and D.~S. Lubinsky.
\newblock Universality limits in the bulk for varying measures.
\newblock {\em Adv. Math.}, 219(3):743--779, 2008.

\bibitem[LS52]{LukacsSzasz}
Eugene Lukacs and Otto Sz{\'a}sz.
\newblock On analytic characteristic functions.
\newblock {\em Pacific J. Math.}, 2:615--625, 1952.

\bibitem[Meh04]{Mehta}
Madan~Lal Mehta.
\newblock {\em Random matrices}, volume 142 of {\em Pure and Applied
  Mathematics (Amsterdam)}.
\newblock Elsevier/Academic Press, Amsterdam, third edition, 2004.

\bibitem[Mon73]{Montgomery}
H.~L. Montgomery.
\newblock The pair correlation of zeros of the zeta function.
\newblock In {\em Analytic number theory ({P}roc. {S}ympos. {P}ure {M}ath.,
  {V}ol. {XXIV}, {S}t. {L}ouis {U}niv., {S}t. {L}ouis, {M}o., 1972)}, pages
  181--193. Amer. Math. Soc., Providence, R.I., 1973.

\bibitem[PS97]{PasturS97}
L.~Pastur and M.~Shcherbina.
\newblock Universality of the local eigenvalue statistics for a class of
  unitary invariant random matrix ensembles.
\newblock {\em J. Statist. Phys.}, 86(1-2):109--147, 1997.

\bibitem[PS08]{PasturS08}
L.~Pastur and M.~Shcherbina.
\newblock Bulk universality and related properties of {H}ermitian matrix
  models.
\newblock {\em J. Stat. Phys.}, 130(2):205--250, 2008.

\bibitem[Sch15]{Schubert}
Kristina Schubert.
\newblock Spacings in orthogonal and symplectic random matrix ensembles.
\newblock {\em Constructive Approximation}, pages 1--38, 2015.

\bibitem[Sos98]{Soshnikov}
Alexander Soshnikov.
\newblock Level spacings distribution for large random matrices: {G}aussian
  fluctuations.
\newblock {\em Ann. of Math. (2)}, 148(2):573--617, 1998.

\bibitem[Tao13]{Tao13}
Terence Tao.
\newblock The asymptotic distribution of a single eigenvalue gap of a {W}igner
  matrix.
\newblock {\em Probab. Theory Related Fields}, 157(1-2):81--106, 2013.

\bibitem[Tit39]{Titchmarsh}
E.C. Titchmarsh.
\newblock {\em The Theory of Functions}.
\newblock Oxford University Press, 2nd edition, 1939.

\bibitem[Tot00]{Totik}
V.~Totik.
\newblock Asymptotics for {C}hristoffel functions with varying weights.
\newblock {\em Adv. in Appl. Math.}, 25(4):322--351, 2000.

\bibitem[Ven13]{Venker13}
Martin Venker.
\newblock Particle systems with repulsion exponent $\beta$ and random matrices.
\newblock {\em Electron. Commun. Probab.}, 18:no. 83, 1--12, 2013.

\bibitem[{Wig}58]{Wigner}
E.~{Wigner}.
\newblock On the distribution of the roots of certain symmetric matrices.
\newblock {\em Ann. of Math. (2)}, 67:325--327, 1958.

\bibitem[Wis28]{Wishart}
John Wishart.
\newblock The generalised product moment distribution in samples from a normal
  multivariate population.
\newblock {\em Biometrika}, 20A(1/2):pp. 32--52, 1928.

\end{thebibliography}

\end{document}